\documentclass{louloupart}
\usepackage{pinlabel}
\usepackage{graphicx}
\title[Virtual an arrow Temperley--Lieb algebras]{Virtual an arrow Temperley--Lieb algebras, Markov traces, and virtual link invariants}
\author[L Paris]{Luis Paris}
\givenname{Luis}
\surname{Paris}
\address{Luis Paris, IMB, UMR 5584, CNRS, Univ. Bourgogne Franche-Comté, 21000 Dijon, France}
\email{lparis@u-bourgogne.fr}
\author[L Rabenda]{Lo\"ic Rabenda}
\givenname{Lo\"ic}
\surname{Rabenda}
\address{Lo\"ic Rabenda, IMB, UMR 5584, CNRS, Univ. Bourgogne Franche-Comté, 21000 Dijon, France}
\email{loic.rabenda@gmail.com}

\newtheorem{thm}{Theorem}[section]
\newtheorem{lem}[thm]{Lemma}
\newtheorem{prop}[thm]{Proposition}
\newtheorem{corl}[thm]{Corollary}

\theoremstyle{definition}
\newtheorem*{rem}{Remark}
\newtheorem*{expl}{Example}
\newtheorem*{acknow}{Acknowledgments}
\newtheorem*{interpret}{Interpretation}

\numberwithin{equation}{section}

\makeatletter
\renewcommand{\thefigure}{\ifnum \c@section>\z@ \thesection.\fi
 \@arabic\c@figure}
\@addtoreset{figure}{section}
\makeatother

\begin{document}

\def\R{\mathbb R} \def\SSS{\mathfrak S} \def\VB{{\rm VB}} 
\def\VV{\mathcal V} \def\LL{\mathcal L} \def\Z{\mathbb Z}
\def\EE{\mathcal E} \def\VTL{{\rm VTL}} \def\BB{\mathcal B}
\def\UU{\mathcal U} \def\N{\mathbb N} \def\MM{\mathcal M}
\def\TL{{\rm TL}} \def\ZZ{\mathcal Z} \def\ATL{{\rm ATL}}
\def\FF{\mathcal F} \def\Span{{\rm Span}}


\begin{abstract}
Let $R^f = \Z [A^{\pm 1}]$ be the algebra of Laurent polynomials in the variable $A$ and let $R^a = \Z[A^{\pm 1}, z_1, z_2, \dots]$ be the algebra of Laurent polynomials in the variable $A$ and standard polynomials in the variables $z_1, z_2, \dots$.
For $n \ge 1$ we denote by $\VB_n$ the virtual braid group on $n$ strands.
We define two towers of algebras $\{ \VTL_n (R^f) \}_{n=1}^\infty$ and $\{ \ATL_n (R^a) \}_{n=1}^\infty$ in terms of diagrams.
For each $n \ge 1$ we determine presentations for both, $\VTL_n (R^f)$ and $\ATL_n (R^a)$.
We determine sequences of homomorphisms $\{ \rho_n^f : R^f [\VB_n] \to \VTL_n (R^f) \}_{n=1}^\infty$ and $\{ \rho_n^a : R^a [\VB_n] \to \ATL_n (R^a) \}_{n=1}^\infty$, we determine Markov traces $\{T_n'^f: \VTL_n (R^f) \to R^f \}_{n=1}^\infty$ and $\{ T_n'^a : \ATL_n (R^a) \to R^a \}_{n=1}^\infty$, and we show that the invariants for virtual links obtained from these Markov traces are the $f$-polynomial for the first trace and the arrow polynomial for the second trace. 
We show that, for each $n \ge 1,$ the standard Temperley--Lieb algebra $\TL_n$ embeds into both, $\VTL_n (R^f)$ and $\ATL_n (R^a)$, and that the restrictions to $\{ \TL_n\}_{n=1}^\infty$ of the two Markov traces coincide.
\end{abstract}

\maketitle


\section{Introduction}\label{sec1}

Let $S_1,\dots,S_\ell$ be a collection of $\ell$ oriented circles smoothly immersed in the plane and having only normal double crossings.
We assign to each crossing a value ``positive'', ``negative'', or ``virtual'', that we indicate on the graphical representation of $S_1 \cup \cdots \cup S_\ell$ as in Figure \ref{fig1_1}.
Such a figure is called a \emph{virtual link diagram}.
We consider the equivalence relation on the set of virtual link diagrams generated by isotopy and the so-called \emph{Reidemeister virtual moves}, as described in Kauffman \cite{Kauff1,Kauff2}.
An equivalence class of virtual link diagrams is called a \emph{virtual link}.

\begin{figure}[ht!]
\begin{center}
\includegraphics[width=6cm]{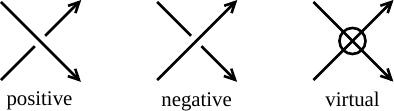}
\caption{Crossings in a virtual link diagram}\label{fig1_1}
\end{center}
\end{figure}

Let $b = (b_1, \dots, b_n)$ be a collection of $n$ smooth paths in the plane $\R^2$ satisfying the following properties.
\begin{itemize}
\item[(a)]
$b_{i}(0)=(0,i)$ for all $i\in\{1,\dots,n\}$, and there exists a permutation $w \in \SSS_n$ such that $b_{i}(1) = (1,w(i))$ for all $i \in \{1, \dots, n\}$.
\item[(b)]
Let $p_1: \R^2 \to \R$ be the projection on the first coordinate.
Then $p_1 (b_i(t)) = t$ for all $i \in \{1, \dots, n\}$ and all $t \in [0,1]$.
\item[(c)]
The union of the images of the $b_i$'s has only normal double crossings. 
\end{itemize}
As with virtual link diagrams, we assign to each crossing a value ``positive'', ``negative'', or ``virtual'', that we indicate on the graphical representation as in Figure \ref{fig1_1}.
Such a figure is called a \emph{virtual braid diagram} on $n$ strands. 
We consider the equivalence relation on the set of virtual braid diagrams on $n$ strands generated by isotopy and some \emph{Reidemeister virtual moves}, as described in Kauffman \cite{Kauff1}.
An equivalence class of virtual braid diagrams on $n$ strands is called a \emph{virtual braid} on $n$ strands.
The virtual braids on $n$ strands form a group, denoted $\VB_n$, called \emph{virtual braid group} on $n$ strands.
The group operation is induced by the concatenation.

We know from Kamada \cite{Kamad1} and Vershinin \cite{Versh1} that $\VB_n$ admits a presentation with generators $\sigma_1, \dots, \sigma_{n-1}, \tau_1, \dots, \tau_{n-1}$ and relations 
\begin{gather*}
\tau_i^2 = 1 \quad\text{for }1\le i \le n-1\,,\\
\sigma_i \sigma_j = \sigma_j \sigma_i\,,\ \tau_i \tau_j = \tau_j \tau_i\,,\ \tau_i \sigma_j = \sigma_j \tau_i \quad \text{for } |i-j| \ge 2\,,\\
\sigma_i \sigma_j \sigma_i = \sigma_j \sigma_i \sigma_j\,,\ \tau_i \tau_j \tau_i = \tau_j \tau_i\tau_j\,,\ \tau_i \tau_j \sigma_i = \sigma_j \tau_i \tau_j \quad \text{for } |i-j| = 1\,.
\end{gather*}
The generators $\sigma_i$ and $\tau_i$ are illustrated in Figure \ref{fig1_2}.

\begin{figure}[ht!]
\begin{center}
\includegraphics[width=4.2cm]{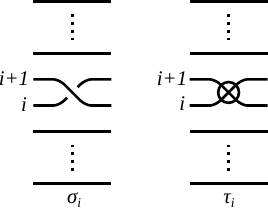}
\caption{Generators of $\VB_n$}\label{fig1_2}
\end{center}
\end{figure}

Note that the subgroup of $\VB_n$ generated by $\sigma_1, \dots, \sigma_{n-1}$ is the braid group $B_n$ on $n$ strands. 
On the other hand, $\VB_n$ may be viewed as a subgroup of $\VB_{n+1}$ via the monomorphism $\VB_n \hookrightarrow \VB_{n+1}$ which sends $\sigma_i$ to $\sigma_i$ and $\tau_i$ to $\tau_i$ for all $i \in \{1, \dots, n-1\}$.

Using the same procedure as for classic braids, we can close a virtual braid $\beta$ and obtain a virtual link, $\hat\beta$, called the \emph{closure} of $\beta$.
We know that each virtual link is the closure of a virtual braid, and we can say when two closed virtual braids are equivalent in terms of virtual Markov moves, as follows.

We denote by $\VB = \bigsqcup_{n=1}^\infty \VB_n$ the disjoint union of all virtual braid groups.
Let $\beta_1, \beta_2 \in \VB$.
We say that $\beta_1$ and $\beta_2$ are connected by a \emph{virtual Markov move} if we are in one of the following four cases. 
\begin{itemize}
\item[(a)]
There exist $n \ge 1$ and $\alpha \in \VB_n$ such that $\beta_1, \beta_2 \in \VB_n$ and $\beta_2 = \alpha \beta_1 \alpha^{-1}$.
\item[(b)]
There exist $n \ge 1$ and $u \in \{\sigma_n, \sigma_n^{-1}, \tau_n\}$ such that $\beta_1 \in \VB_n$, $\beta_2 \in \VB_{n+1}$ and $\beta_2 = \beta_1 u$, or vice versa.
\item[(c)]
There exists $n \ge 2$ such that $\beta_1 \in \VB_n$, $\beta_2 \in \VB_{n+1}$, and $\beta_2 = \beta_1 \sigma_n^{-1} \tau_{n-1} \sigma_n$, or vice versa.
\item[(d)]
There exists $n \ge 2$ such that $\beta_1 \in \VB_n$, $\beta_2 \in \VB_{n+1}$, and $\beta_2 = \beta_1 \tau_n \tau_{n-1} \sigma_{n-1} \tau_n \sigma_{n-1}^{-1} \tau_{n-1} \tau_n$, or vice versa.
\end{itemize}

\begin{thm}[Kamada \cite{Kamad1}, Kauffman--Lambropoulou \cite{KauLam1}]\label{thm1_1}
Let $\beta_1, \beta_2 \in \VB$.
Then $\hat \beta_1 = \hat \beta_2$ if and only if $\beta_1$ and $\beta_2$ are connected by a finite sequence of virtual Markov moves.
\end{thm}

Let $R$ be a ring. 
For $n \ge 1$ we denote by $R [\VB_n]$ the group $R$-algebra of $\VB_n$.
Notice that, since $\VB_n$ is a subgroup of $\VB_{n+1}$, $R [\VB_n]$ is a subalgebra of $R[\VB_{n+1}]$.
A sequence $\{T_n: R [\VB_n] \to R\}_{n=1}^\infty$ of $R$-linear forms is called a \emph{Markov trace} if it satisfies the following properties. 
\begin{itemize}
\item[(a)]
$T_n (x y) = T_n (y x)$ for all $n \ge 1$ and all $x,y \in R[\VB_n]$.
\item[(b)]
$T_n (x) = T_{n+1} (x \sigma_n) = T_{n+1} (x \sigma_n^{-1}) = T_{n+1} (x \tau_n)$ for all $n \ge 1$ and all $x \in R [\VB_n]$.
\item[(c)]
$T_n (x) = T_{n+1} (x \sigma_n^{-1} \tau_{n-1} \sigma_n)$ for all $n \ge 2$ and all $x \in R [\VB_n]$.
\item[(d)]
$T_n (x) = T_{n+1} (x \tau_n \tau_{n-1} \sigma_{n-1} \tau_n \sigma_{n-1}^{-1} \tau_{n-1} \tau_n)$ for all $n \ge 2$ and all $x \in R [\VB_n]$.
\end{itemize}

Note that our definition of ``Markov trace'' is not the one that can be usually found in the literature (see Kauffman--Lambropoulou \cite{KauLam1}, for example), but all know definitions, including this one, are equivalent up to renormalization. 

Let $\VV\LL$ be the set of virtual links.
Thanks to Theorem \ref{thm1_1}, from a Markov trace $\{T_n: R [\VB_n] \to R\}_{n=1}^\infty$ we can define an invariant $I: \VV\LL \to R$ by setting $I (\hat \beta) = T_n (\beta)$ for all $n \ge 1$ and all $\beta \in \VB_n$.
Conversely, any invariant of virtual links with coefficients in $R$ can be obtained in this way. 

A \emph{tower of algebras} is a sequence $\{A_n\}_{n=1}^\infty$ of algebras such that $A_n$ is a subalgebra of $A_{n+1}$ for all $n \ge 1$.
A sequence $\{\rho_n: R[\VB_n] \to A_n\}_{n=1}^\infty$ of homomorphisms is said to be \emph{compatible} if the restriction of $\rho_{n+1}$ to $R[\VB_n]$ is equal to $\rho_n$ for all $n$.
Let $\{A_n\}_{n=1}^\infty$ be a tower of algebras and let $\{\rho_n: R[\VB_n] \to A_n\}_{n=1}^\infty$ be a compatible sequence of homomorphisms.
Set $S_i = \rho_n (\sigma_i)$ and $v_i = \rho_n (\tau_i)$ for all $i \in \{1, \dots, n-1\}$.
A sequence $\{T_n': A_n \to R\}_{n=1}^\infty$ of linear forms is a \emph{Markov trace} if it satisfies the following properties.
\begin{itemize}
\item[(a)]
$T_n'(xy)=T_n'(yx)$ for all $n \ge 1$ and all $x,y \in A_n$.
\item[(b)]
$T_n'(x)=T_{n+1}' (x S_n) = T_{n+1}'(x S_n^{-1}) = T_{n+1}' (x v_n)$ for all $n \ge 1$ and all $x \in A_n$. 
\item[(c)]
$T_n'(x)=T_{n+1}'(x S_n^{-1} v_{n-1} S_n)$ for all $n \ge 2$ and all $x \in A_n$.
\item[(d)]
$T_n' (x) = T_{n+1}' (x v_n v_{n-1} S_{n-1} v_n S_{n-1}^{-1} v_{n-1} v_n)$ for all $n \ge 2$ and all $x \in A_n$.
\end{itemize}
Clearly, in that case, the sequence $\{T_n = T_n' \circ \rho_n : R [\VB_n] \to R\}_{n=1}^\infty$ is a Markov trace, and therefore it determines an invariant for virtual links.

A ``natural'' strategy to build Markov traces on $\{K[\VB_n]\}_{n=1}^\infty$, and therefore invariants for virtual links, would be to transit through Markov traces on compatible towers of algebras, as defined above. 
This strategy won its spurs in the classical theory of knots and links, in particular thanks to Jones' definitions of the Jones polynomial \cite{Jones2} and of the HOMFLY-PT polynomial \cite{Jones1}.
As far as we know, this strategy is poorly used in the theory of virtual knots and links.
Actually, the only reference we found is Li--Lei--Li \cite{LiLeLi1}, where the authors define a tower of algebras in terms of diagrams, claim (with no proof) that their algebras are the same as the virtual Temperley--Lieb algebras of Zhang--Kauffman--Ge \cite{ZhKAGe1}, and show that the $f$-polynomial can be obtained from a Markov trace on this tower of algebras. 
By the way, one of their main results, \cite[Proposition 4.1]{LiLeLi1}, is wrong (see Proposition \ref{prop2_3}).

Our aim in the present paper is to describe two invariants for virtual links in terms of Markov traces: the $f$-polynomial, also known as the Jones--Kauffman polynomial, and the arrow polynomial.
The $f$-polynomial is a version of the Jones polynomial for virtual links defined from the Kauffman bracket.
This was introduced by Kauffman \cite{Kauff1} in his seminal paper on virtual knots and links, and its construction closely follows Kauffman's construction \cite{Kauff4} of the Jones polynomial for classical links.
The arrow polynomial is a refinement of the $f$-polynomial.
It coincides with the Jones polynomial on classical links, but it is much more powerful for (non-classical) virtual links.
In particular, it provides a lower bound for the number of virtual crossings.
It was constructed by Miyazawa \cite{Miyaz1} and Dye--Kauffman \cite{DyeKau1} (see also Kauffman \cite{Kauff3}).

Section \ref{sec2} is dedicated to the construction of a Markov trace associated with the $f$-polynomial.
Our approach is  close to that of Li--Lei--Li \cite{LiLeLi1}, but, on the one hand, our study of the $f$-polynomial is needed in our study of the arrow polynomial, and, on the other hand, we complete the study of Li--Lei--Li \cite{LiLeLi1} with correct presentations for virtual Temperley--Lieb algebras and other results.
For each $n \ge 1$ we define an algebra $\VTL_n (R^f)$ in terms of diagrams, so that the sequence $\{ \VTL_n (R^f)\}_{n=1}^\infty$ is a tower of algebras. 
In Proposition \ref{prop2_2} we determine a presentation for $\VTL_n (R^f)$ and in Proposition \ref{prop2_3} we show that the presentation for $\VTL_n (R^f)$ given in Li--Lei--Li \cite{LiLeLi1} is wrong. 
Actually, the relation $E_i E_j E_i= E_i$ for $|i-j|=1$, which is standard in Temperley--Lieb algebras, must be replaced by a ``virtual relation'' of the form $E_i v_j E_i = E_i$.
Then we determine a compatible sequence of homomorphisms $\{ \rho_n^f : R^f [\VB_n] \to \VTL_n (R^f) \}_{n=1}^\infty$ (Theorem \ref{thm2_5}), we determine a Markov trace on the tower of algebras $\{ \VTL_n (R^f) \}_{n=1}^\infty$ (Theorem \ref{thm2_6}), and we show that this construction leads to the $f$-polynomial (Theorem \ref{thm2_8}).

Section \ref{sec3} is dedicated to the arrow polynomial.
Our construction can be viewed as a labeled version of the construction of Section \ref{sec2}.
For each $n \ge 1$ we define an algebra $\ATL_n (R^a)$ in terms of labeled diagrams, so that $\{ \ATL_n (R^a)\}_{n=1}^\infty$ is a tower of algebras. 
Intuitively speaking, a label represents the number of cusps in Kauffman sense that can be found on an arc. 
In Proposition \ref{prop3_2} we determine a presentation for $\ATL_n (R^a)$.
This is a sort of labeled version of the presentation for $\VTL_n (R^f)$ given in Proposition \ref{prop2_2}.
Then we proceed as in Section \ref{sec2}: we determine a compatible sequence of homomorphisms $\{ \rho_n^a : R^a [\VB_n] \to \ATL_n (R^a) \}_{n=1}^\infty$ (Theorem \ref{thm3_4}), we determine a Markov trace on the tower of algebras $\{ \ATL_n (R^a) \}_{n=1}^\infty$ (Theorem \ref{thm3_5}), and we show that this construction leads to the arrow polynomial (Theorem \ref{thm3_9}).

It is known that the arrow polynomial coincides with the $f$-polynomial on classical links (see Miyazawa \cite{Miyaz1} and Dye--Kauffman \cite{DyeKau1}).
We show that this fact has an interpretation in terms of Markov traces on Temperley--Lieb algebras.
For $n \ge 1$ we denote by $\TL_n$ the $n$-th standard Temperley--Lieb algebra.
We show that $\TL_n$ embeds into both, $\VTL_n(R^f)$ (Proposition \ref{prop2_4}) and $\ATL_n (R^a)$ (Proposition \ref{prop3_3}), and that the restriction to $\{ \TL_n\}_{n=1}^\infty$ of the Markov trace on $\{ \VTL_n (R^f) \}_{n=1}^\infty$ coincides with the restriction to $\{ \TL_n\}_{n=1}^\infty$ of the Markov trace on $\{ \ATL_n (R^a) \}_{n=1}^\infty$ (Proposition \ref{prop3_7}).

\begin{acknow}
The first author is supported by the French project ``AlMaRe'' (ANR-19-CE40-0001-01) of the ANR.
\end{acknow}


\section{Virtual Temperley--Lieb algebras and f-polynomial}\label{sec2}

Two rings are involved in this section. 
The first is the ring $R_0^f = \Z [z]$ of polynomials in the variable $z$ with integer coefficients. The second is the ring $R^f = \Z [A^{\pm 1}]$ of Laurent polynomials in the variable $A$ with integer coefficients. 
We assume that $R_0^f$ is embedded into $R^f$ via the identification $z=-A^2-A^{-2}$.
Notice that the superscript $f$ over $R_0$ and $R$ in this notation is to underline the fact that all the constructions in the present section concern the $f$-polynomial. 
In contrast, the rings in the next section, which concerns the arrow polynomial, will be denoted $R_0^a$ and $R^a$.  

We start recalling the definition of the $f$-polynomial, as it will help the reader to understand the constructions and definitions that follow after.

Define the \emph{Kauffman bracket} $\langle L \rangle \in R^f$ of a (non-oriented) virtual link diagram $L$ as follows. 
If $L$ has only virtual crossings, then $\langle L \rangle = z^\ell = (-A^2-A^{-2})^\ell$, where $\ell$ is the number of components of $L$.
Suppose that $L$ has at least one non-virtual crossing $p$.
Then $\langle L \rangle = A\,\langle L_1 \rangle + A^{-1}\,\langle L_2 \rangle$, where $L_1$ and $L_2$ are identical to $L$ except in a neighborhood of $p$ where there are as shown in Figure \ref{fig2_1}.
The \emph{writhe} of an (oriented) virtual link diagram $L$, denoted $w (L)$, is the number of positive crossings minus the number of negative crossings. 
Then the \emph{$f$-polynomial} of an (oriented) virtual link diagram $L$ is $f(L) = f(L) (A) = (-A^3)^{-w(L)} \langle L \rangle$.

\begin{figure}[ht!]
\begin{center}
\includegraphics[width=6cm]{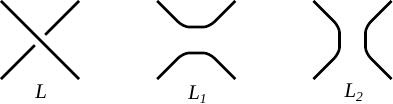}
\caption{Relation in the Kauffman bracket}\label{fig2_1}
\end{center}
\end{figure}

\begin{thm}[Kauffman \cite{Kauff1}]\label{thm2_1}
If two virtual link diagrams $L$ and $L'$ are equivalent, then $f(L)=f(L')$.
\end{thm}

The \emph{$f$-polynomial} of a virtual link $L$, denoted $f(L)$, is the $f$-polynomial of any of its diagrams.
This is a well-defined invariant thanks to Theorem \ref{thm2_1}.

Our goal now is to define a Markov trace whose associated invariant is the $f$-polynomial. 
We proceed as indicated in the introduction: we pass through a tower of algebras, the tower of virtual Temperley--Lieb algebras.

Let $n\ge 1$ be an integer. 
A \emph{flat virtual $n$-tangle} is a collection of $n$ disjoint pairs in $\{0,1\}\times\{1,\dots,n\}$, that is, a partition of $\{0,1\}\times\{1,\dots,n\}$ into pairs.
Let $E=\{\alpha_1, \dots, \alpha_n\}$ be a flat virtual $n$-tangle.
Then we graphically represent $E$ on the plane by connecting the two ends of each $\alpha_i$ with an arc. 
For example, Figure \ref{fig2_2} represents the flat virtual $3$-tangle $\{\alpha_1, \alpha_2, \alpha_3\}$, where $\alpha_1=\{(0,1), (0,2)\}$, $\alpha_2=\{ (0,3), (1,2)\}$ and $\alpha_3=\{(1,3), (1,1)\}$.

\begin{figure}[ht!]
\begin{center}
\includegraphics[width=2.2cm]{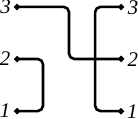}
\caption{Flat virtual tangle}\label{fig2_2}
\end{center}
\end{figure}

We denote by $\EE_n$ the set of flat virtual $n$-tangles, and by $\VTL_n$ the free $R_0^f$-module freely generated by $\EE_n$.
We define a multiplication in $\VTL_n$ as follows.
Let $E$ and $E'$ be two flat virtual $n$-tangles. 
By concatenating the diagrams of $E$ and $E'$ we get a family of closed curves and $n$ arcs. 
These $n$ arcs determine a partition of $\{0,1\}\times\{1,\dots,n\}$ into pairs, that is, a flat virtual $n$-tangle that we denote by $E*E'$.
Let $m$ be the number of obtained closed curves. 
Then we set $E\, E'=z^m\,(E*E')$.
It is easily checked that $\VTL_n$ endowed with this multiplication is an (unitary and associative) algebra that we call the $n$-th \emph{virtual Temperley--Lieb algebra}.

\begin{expl}
On the left hand side of Figure \ref{fig2_3} are illustrated the diagrams of two flat virtual $4$-tangles, $E$ and $E'$, and on the right hand side a diagram of $E*E'$.
In this example, by concatenating the diagrams of $E$ and $E'$ we get only one closed curve, hence $m=1$ and $E\,E'=z(E*E')$.
\end{expl}

\begin{figure}[ht!]
\begin{center}
\includegraphics[width=7.8cm]{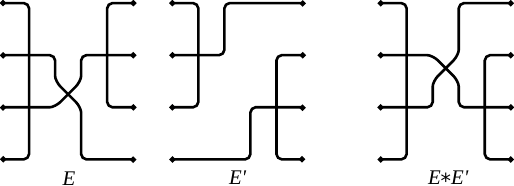}
\caption{Multiplication in $\VTL_n$}\label{fig2_3}
\end{center}
\end{figure}

\begin{rem}
It is easily seen that the embedding $\EE_n \to \EE_{n+1}$ which sends each $E \in \EE_n$ to $E \cup \{ \{ (0,n+1), (1,n+1) \} \}$ induces an injective homomorphism $\VTL_n \hookrightarrow \VTL_{n+1}$, for all $n \ge 1$.
So, we have a tower of algebras $\{ \VTL_n \}_{n=1}^\infty$.
\end{rem}

\begin{prop}\label{prop2_2}
Let $n\ge 2$.
Then $\VTL_n$ has a presentation with generators $E_1,\dots, E_{n-1},v_1, \dots, v_{n-1}$ and relations
\begin{gather*}
E_i^2=z E_i\,,\ v_i^2=1\,,\ E_iv_i=v_iE_i=E_i\,,\quad\text{for } 1 \le i \le n-1\,,\\
E_i E_j = E_j E_i\,,\ v_i v_j = v_j v_i \,,\ v_i E_j = E_j v_i\,,\quad\text{for } |i-j| \ge 2\,,\\
E_i v_j E_i = E_i\,,\ v_i v_j v_i = v_j v_i v_j\,,\ v_i v_j E_i = E_j v_i v_j \,,\quad \text{for } |i-j| = 1\,.
\end{gather*}
\end{prop}

The generators $E_i$ and $v_i$ are illustrated in Figure \ref{fig2_4}.

\begin{figure}[ht!]
\begin{center}
\includegraphics[width=4.2cm]{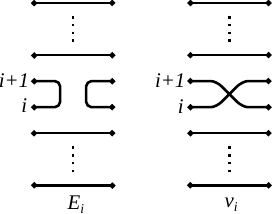}
\caption{Generators of $\VTL_n$}\label{fig2_4}
\end{center}
\end{figure}

\begin{proof}
Let $A_n$ be the algebra over $R_0^f$ defined by the presentation with generators $X_1, \dots, X_{n-1}, y_1, \dots,
\allowbreak
y_{n-1}$ and relations
\begin{gather*}
X_i^2=z X_i\,,\ y_i^2=1\,,\ X_iy_i=y_iX_i=X_i\,,\quad\text{for }1\le i \le n-1\,,\\
X_iX_j=X_jX_i\,,\ y_iy_j=y_jy_i\,,\ y_iX_j=X_jy_i\,,\quad\text{for }|i-j|\ge2\,,\\
X_iy_jX_i=X_i\,,\ y_iy_jy_i=y_jy_iy_j\,,\ y_iy_jX_i=X_jy_iy_j\,,\quad\text{for }|i-j|=1\,.
\end{gather*}
It is easily checked using diagrammatic calculation that there is a homomorphism $\varphi: A_n \to \VTL_n$ which sends $X_i$ to $E_i$ and $y_i$ to $v_i$ for all $i \in \{1,\dots,n-1\}$.
We need to prove that $\varphi$ is an isomorphism. 

{\it Claim 1.}
{\it The following relations hold in $A_n$.
\begin{gather*}
X_iX_jy_iy_j=X_i \quad\text{for }|i-j|=1\,,\\
X_iX_jX_i=X_i\quad\text{for }|i-j|=1\,,\\
X_iX_j=y_jy_iX_j\quad\text{for }|i-j|=1\,,\\
y_{i+1}y_{i+2}y_iy_{i+1}X_iX_{i+2}=X_iX_{i+2}\quad\text{for }1\le i\le n-3\,.
\end{gather*}}

{\it Proof of Claim 1.}
Let $i,j\in\{1,\dots,n-1\}$ such that $|i-j|=1$.
Then
\begin{gather*}
X_iX_jy_iy_j=
X_iy_iy_jX_i=
X_iy_jX_i=
X_i\,,\\
X_iX_jX_i=
X_iX_jy_iy_jy_jy_iX_i=
X_iy_jy_iX_i=
X_iy_jX_i=
X_i\,,\\
X_iX_j=
y_jy_iy_iy_jX_iX_j=
y_jy_iX_jy_iy_jX_j=
y_jy_iX_jy_iX_j=
y_jy_iX_j\,.
\end{gather*}
Let $i\in\{1,\dots,n-3\}$.
Then
\begin{gather*}
y_{i+1}y_{i+2}y_iy_{i+1}X_iX_{i+2}=
y_{i+1}y_{i+2}X_{i+1}X_iX_{i+2}=
X_{i+2}X_{i+1}X_iX_{i+2}=\\
X_{i+2}X_{i+1}X_{i+2}X_i=
X_{i+2}X_i=
X_iX_{i+2}\,.
\end{gather*}
This concludes the proof of Claim 1. 

We denote by $\UU(A_n)$ the group of units of $A_n$ and by $\SSS_n$ the $n$-th symmetric group. 
We have a homomorphism $\iota:\SSS_n\to\UU(A_n)$ which sends $s_i$ to $y_i$ for all $i\in\{1,\dots,n-1\}$, where $s_i=(i,i+1)$.
Let $n_0$ be the integer part of $\frac{n}{2}$.
Let $p\in\{0,1,\dots,n_0\}$. 
We set $B_p=X_1 X_3 \cdots X_{2p-1}$ if $p\neq 0$, and $B_p = B_0 = 1$ if $p=0$. 
We denote by $U_{1,p}$ the set of $w \in \SSS_n$ satisfying 
\begin{gather*}
w(1) < w(3) < \cdots < w(2p-1)\,,\\
w(2i-1) <w(2i) \quad \text{for } 1 \le i \le p\,,\\
w(2p+1) < w(2p+2) < \cdots < w(n)\,,
\end{gather*}
and we denote by $U_{2,p}$ the set of $w \in \SSS_n$ satisfying 
\begin{gather*}
w(1) < w(3) < \cdots <w(2p-1)\,,\\
w(2i-1) <w(2i) \quad \text{for } 1 \le i \le p\,.
\end{gather*}
Then we set 
\[
\BB_p = \{ \iota(w_1) \, B_p \, \iota(w_2^{-1}) \mid w_1 \in U_{1,p},\ w_2 \in U_{2,p} \}\,,
\]
for $0 \le p \le n_0$, and 
\[
\BB = \bigcup_{p=0}^{n_0} \BB_p\,.
\]

Let $E \in \EE_n$.
We can write $E$ in the form $E= \{ \alpha_1, \dots, \alpha_p, \alpha_1', \dots, \alpha_p',\beta_1, \dots, \beta_q\}$, where 
\begin{itemize}
\item
each $\alpha_i$ is of the form $\alpha_i= \{(0,a_i),(0,b_i)\}$, with $1 \le a_1 <a_2 < \cdots <a_p\le n$, and $a_i <b_i$ for all $i\in\{1,\dots,p\}$;
\item
each $\alpha_i'$ is of the form $\alpha_i'= \{(1,a_i'),(1,b_i')\}$, with $1 \le a_1'<a_2'<\cdots<a_p'\le n$, and $a_i'<b_i'$ for all $i\in\{1,\dots,p\}$;
\item
each $\beta_j$ is of the form $\beta_j=\{(0,c_j),(1,d_j)\}$, with $1 \le c_1<c_2<\cdots<c_q\le n$, and $2p+q=n$.
\end{itemize}
We define $w_1\in\SSS_n$ by $w_1(2i-1)=a_i$ and $w_1(2i)=b_i$ for all $i\in\{1,\dots,p\}$, and $w_1(2p+j)=c_j$ for all $j\in\{1,\dots,q\}$.
Similarly, we define $w_2\in\SSS_n$ by $w_2(2i-1)=a_i'$ and $w_2(2i)=b_i'$ for all $i\in\{1,\dots,p\}$, and $w_2(2p+j)=d_j$ for all $j\in\{1,\dots,q\}$.
We see that $w_1 \in U_{1,p}$, $w_2 \in U_{2,p}$, and $E=\varphi(\iota(w_1)\,B_p\,\iota(w_2^{-1}))$.
Moreover, such a form is unique for each $E \in \EE_n$, and we have $\varphi(Y) \in \EE_n$ for all $Y \in \BB$, hence $\varphi$ restricts to a bijection from $\BB$ to $\EE_n$.

So, in order to prove Proposition \ref{prop2_2}, it suffices to show that $\BB$ spans $A_n$ as a $R_0^f$-module.
We denote by $\MM$ the submonoid of $A_n$ generated by $X_1, \dots, X_{n-1}, y_1, \dots, y_{n-1}$, that is, the set of finite products of elements in $\{X_1, \dots, X_{n-1}, y_1, \dots, y_{n-1} \}$.
By definition $\MM$ spans $A_n$ as a $R_0^f$-module, hence we only need to show that $\MM$ is contained in the $R_0^f$-submodule $\Span_{R_0^f} (\BB)$ of $A_n$ spanned by $\BB$.

{\it Claim 2.}
{\it Let $w_1, w_2 \in \SSS_n$ and $p \in \{0, 1, \dots, n_0\}$.
Then $\iota (w_1) \, B_p\, \iota(w_2^{-1}) \in \BB$.}

{\it Proof of Claim 2.}
Let $i \in \{1, \dots, p\}$ such that  $w_1 (2i-1) = b_i > w_1(2i) = a_i$.
By applying the relation $y_{2i-1} X_{2i-1} = X_{2i-1}$ we can replace $w_1$ with $w_1 s_{2i-1}$, and then $w_1 (2i-1) = a_i < w_1 (2i) = b_i$.
So, we can assume that $w_1 (2i-1) < w_1 (2i)$ for all $i \in \{1, \dots, p\}$.
Let $i \in \{1, \dots, p-1\}$ such that $w_1 (2i-1) = a_{i+1} > w_1 (2i+1) = a_i$.
By applying the relation $y_{2i} y_{2i-1} y_{2i+1} y_{2i} X_{2i-1} X_{2i+1} = X_{2i-1} X_{2i+1}$ we can replace $w_1$ with $w_1 s_{2i} s_{2i-1} s_{2i+1} s_{2i}$, and then we have $w_1 (2i-1) = a_i < w_1 (2i+1) = a_{i+1}$ while keeping the inequalities $w_1 (2i-1) < w_1 (2i)$ and $w_1 (2i+1) < w_1 (2i+2)$.
So, we can also assume that $w_1 (1) < w_1 (3) < \cdots <w_1 (2p-1)$.
Set $q = n-2p$.
Let $j \in \{1, \dots, q-1\}$ such that $w_1 (2p+j) = c_{j+1} >w_1 (2p+j+1) = c_j$.
By applying the relations $y_{2p+j} X_{2i-1} = X_{2i-1} y_{2p+j}$ for $i \in \{1, \dots, p\}$ we can replace $w_1$ with $w_1 s_{2p+j}$ and $w_2$ with $w_2 s_{2p+j}$, and then $w_1 (2p+j) = c_j <w_1 (2p+j+1) = c_{j+1}$.
So, we can also assume that $w_1 (2p+1) < w_1 (2p+2) < \cdots < w_1 (n)$, that is, $w_1 \in U_{1,p}$.
We use the same argument to show that $w_2$ can be replaced with some $w_2' \in U_{2,p}$.
So, $\iota (w_1) \, B_p \,\iota(w_2^{-1}) \in \BB$.
This concludes the proof of Claim 2. 

{\it Claim 3.}
{\it Let $a, b \in \{1,\dots, n-1\}$, $a \le b$, and $p \in \{0, 1, \dots, n_0\}$.
Then $X_a y_{a+1} \cdots y_b B_p \in \Span_{R_0^f} (\BB)$.}

{\it Proof of Claim 3.}
Suppose $a \ge 2p+1$ (which is always true if $p=0$). 
Then 
\begin{gather*}
X_a y_{a+1} \cdots y_b B_p=
X_a B_p y_{a+1} \cdots y_b=\\
X_a (y_{a-1} y_a) (y_{a-2} y_{a-1}) \cdots (y_{2p+1} y_{2p+2}) (y_{2p+2} y_{2p+1}) \cdots (y_{a-1} y_{a-2}) (y_a y_{a-1}) B_p y_{a+1} \cdots y_b =\\
(y_{a-1} y_a) (y_{a-2} y_{a-1}) \cdots (y_{2p+1} y_{2p+2}) X_{2p+1} B_p (y_{2p+2} y_{2p+1}) \cdots (y_{a-1} y_{a-2}) (y_a y_{a-1}) y_{a+1} \cdots y_b =\\
(y_{a-1} y_a) (y_{a-2} y_{a-1}) \cdots (y_{2p+1} y_{2p+2}) B_{p+1} (y_{2p+2} y_{2p+1}) \cdots (y_{a-1} y_{a-2}) (y_a y_{a-1}) y_{a+1} \cdots y_b =\\
\iota(w_1)\, B_{p+1} \iota(w_2^{-1})
\in \BB\,,
\end{gather*}
where $w_1 = (s_{a-1} s_{a}) \cdots (s_{2p+1} s_{2p+2})$ and $w_2 = s_b \cdots s_{a+1} (s_{a-1} s_{a}) \cdots (s_{2p+1} s_{2p+2})$.
Suppose $a \le 2p$, $a$ is odd, and $a=b$. 
Let $c \in \{1, \dots, p\}$ such that $a=2c-1$.
Then 
\begin{gather*}
X_a y_{a+1} \cdots y_b B_p =
X_{2c-1}^2 X_1 \cdots X_{2c-3} X_{2c+1} \cdots X_{2p-1} =\\
zX_{2c-1} X_1 \cdots X_{2c-3} X_{2c+1} \cdots X_{2p-1} =
z B_p \in \Span_{R_0^f} (\BB)\,.
\end{gather*}
Suppose $a \le 2p$, $a$ is odd, and $a < b$. 
Let $c \in \{1, \dots, p\}$ such that $a=2c-1$.
Then 
\begin{gather*}
X_a y_{a+1} \cdots y_b B_p =
X_{2c-1} y_{2c} \cdots y_b X_{2c-1} X_1 \cdots X_{2c-3} X_{2c+1} \cdots X_{2p-1} =\\
X_{2c-1} y_{2c} X_{2c-1} y_{2c+1} \cdots y_b X_1 \cdots X_{2c-3} X_{2c+1} \cdots X_{2p-1} =\\
X_{2c-1} y_{2c+1} \cdots y_b X_1 \cdots X_{2c-3} X_{2c+1} \cdots X_{2p-1} =\\
y_{2c+1} \cdots y_b X_1 \cdots X_{2c-3} X_{2c-1} X_{2c+1} \cdots X_{2p-1} =
\iota(s_{2c+1} \cdots s_b)\, B_p
\in \BB\,.
\end{gather*}
Suppose $a \le 2p$ and $a$ is even.  
Let $c \in \{1, \dots, p\}$ such that $a=2c$.
Then 
\begin{gather*}
X_a y_{a+1} \cdots y_b B_p =
X_{2c} y_{2c+1} \cdots y_b X_{2c-1} X_1 \cdots X_{2c-3} X_{2c+1} \cdots X_{2p-1} =\\
X_{2c} X_{2c-1} y_{2c+1} \cdots y_b X_1 \cdots X_{2c-3} X_{2c+1} \cdots X_{2p-1} =\\
y_{2c-1} y_{2c} X_{2c-1} y_{2c+1} \cdots y_b X_1 \cdots X_{2c-3} X_{2c+1} \cdots X_{2p-1} =\\ 
y_{2c-1} y_{2c} y_{2c+1} \cdots y_b X_1 \cdots X_{2c-3} X_{2c-1} X_{2c+1} \cdots X_{2p-1} = 
\iota (s_{2c-1} s_{2c} s_{2c+1} \cdots s_b)\, B_p \in \BB\,.
\end{gather*}
This concludes the proof of Claim 3.

{\it Claim 4.}
{\it Let $p \in \{0, 1, \dots, n_0\}$ and $w \in \SSS_n$.
Then $X_1 \, \iota(w) \, B_p \in \Span_{R_0^f} (\BB)$.}

{\it Proof of Claim 4.}
There exist $a \in \{1, \dots, n-1\}$, $b \in \{0, 1, \dots, n-1\}$ and $w_1 \in \langle s_3, \dots,s_{n-1} \rangle$ such that $w=w_1 s_2 s_3 \cdots s_a s_1 s_2 \cdots s_b$.
Suppose $a \le b$.
Then
\begin{gather*}
X_1 \, \iota(w) \, B_p =
\iota(w_1)\, X_1 y_2 \cdots y_a y_1 \cdots y_{a-1} y_a y_{a+1} \cdots y_b B_p = \\
\iota(w_1) \, X_1 (y_2 y_1) (y_3 y_2) \cdots (y_a y_{a-1}) y_a y_{a+1} \cdots y_b B_p = \\
\iota(w_1)\,(y_2 y_1)(y_3 y_2)\cdots (y_a y_{a-1}) X_a y_a y_{a+1} \cdots y_b B_p = \\
\iota(w_1 (s_2 s_1) (s_3 s_2) \cdots (s_a s_{a-1})) \, X_a y_{a+1} \cdots y_b B_p\,.
\end{gather*}
We know by Claim 3 that $X_a y_{a+1} \cdots y_b B_p \in \Span_{R_0^f} (\BB)$, hence, by Claim 2, $X_1\,\iota(w)\,B_p \in \Span_{R_0^f} (\BB)$.
Suppose $a > b$.
Then 
\begin{gather*}
X_1\,\iota(w)\,B_p=
\iota(w_1)\,X_1 y_2 \cdots y_b y_{b+1} \cdots y_a y_1 \cdots y_b B_p = \\
\iota(w_1)\,X_1 (y_2 y_1) (y_3 y_2) \cdots (y_{b+1} y_b) y_{b+2} \cdots y_a B_p = \\
\iota(w_1)\,(y_2 y_1) (y_3 y_2) \cdots (y_{b+1} y_b) X_{b+1} y_{b+2} \cdots y_a B_p = \\
\iota(w_1 (s_2 s_1) (s_3 s_2) \cdots (s_{b+1} s_b))\, X_{b+1} y_{b+2} \cdots y_a B_p\,.
\end{gather*}
We know by Claim 3 that $X_{b+1} y_{b+2} \cdots y_a B_p \in \Span_{R_0^f} (\BB)$, hence, by Claim 2, $X_1\,\iota(w)\,B_p \in \Span_{R_0^f} (\BB)$. 
This concludes the proof of Claim 4. 

As pointed out above, the following claim ends the proof of Proposition \ref{prop2_2}.

{\it Claim 5.}
{\it The monoid $\MM$ is contained in $\Span_{R_0^f} (\BB)$.}

{\it Proof of Claim 5.}
Let $Y \in \MM$.
By using the relations $y_i y_j X_i = X_j y_i y_j$ for $|i-j| = 1$, we see that $Y$ can be written in the form $Y=\iota(w_0)\,X_1\,\iota(w_1)\,X_1\,\cdots X_1\,\iota(w_k)$ where $k \ge 0$ and $w_0, w_1,\dots,w_k \in \SSS_n$.
We prove that $Y \in \Span_{R_0^f} (\BB)$ by induction on $k$.
The case $k = 0$ is trivial and the case $k = 1$ follows from Claim 2.
So, we can assume that $k \ge 2$ and that the inductive hypothesis holds. 
By the inductive hypothesis $\iota(w_1) \, X_1 \cdots X_1 \, \iota (w_k) \in \Span_{R_0^f} (\BB)$, 
thus we just need to prove that $\iota (w_0) X_1 \iota (w_1') B_p \iota(w_2'^{-1}) \in \Span_{R_0^f} (\BB)$ for all $p \in \{0, 1, \dots, n_0\}$ and all $w_1', w_2' \in \SSS_n$.
We know by Claim 4 that $X_1 \iota (w_1') B_p \in \Span_{R_0^f} (\BB)$, hence, by Claim 2, $\iota (w_0) X_1 \iota (w_1') B_p \iota(w_2'^{-1}) \in \Span_{R_0^f} (\BB)$.
This concludes the proof of Claim 5.
\end{proof}

We have seen that the relation $E_iE_jE_i=E_i$ holds for $|i-j|=1$ in $\VTL_n$ (see Claim 1 in the proof of Proposition \ref{prop2_2}).
However, contrary to what some people may believe, we cannot replace the relation $E_iv_jE_i=E_i$ with the relation $E_iE_jE_i=E_i$ in the presentation of $\VTL_n$. 
Indeed:

\begin{prop}\label{prop2_3}
Let $n \ge 3$ and let $\VTL_n'$ be the algebra over $R_0^f$ defined by the presentation with generators $E_1', \dots, E_{n-1}', v_1', \dots, v_{n-1}'$ and relations
\begin{gather*}
E_i'^2 =z E_i'\,,\ v_i'^2=1\,,\ E_i'v_i'=v_i'E_i'=E_i'\,,\quad\text{for }1\le i \le n-1\,,\\
E_i'E_j'=E_j'E_i'\,,\ v_i'v_j'=v_j'v_i'\,,\ v_i'E_j'=E_j'v_i'\,,\quad\text{for }|i-j|\ge2\,,\\
E_i'E_j'E_i'=E_i'\,,\ v_i'v_j'v_i'=v_j'v_i'v_j'\,,\ v_i'v_j'E_i'=E_j'v_i'v_j'\,,\quad\text{for }|i-j|=1\,.
\end{gather*} 
Let $\varphi: \VTL_n' \to \VTL_n$ be the homomorphism that sends $E_i'$ to $E_i$ and $v_i'$ to $v_i$ for all $i \in \{1, \dots, n-1\}$.
Then $\varphi$ is surjective but not injective.
\end{prop}

\begin{proof}
By definition $E_1, \dots, E_{n-1}, v_1, \dots, v_{n-1}$ belong to the image of $\varphi$.
Since these elements generate $\VTL_n$, the homomorphism $\varphi$ is surjective.
Let $C_2=\{\pm 1\}$ be the cyclic group of order $2$ and let $\Z[C_2]$ be the group algebra of $C_2$.
It is easily checked with the presentation of $\VTL_n'$ that there is a ring homomorphism $\theta: \VTL_n' \to \Z [C_2]$ satisfying $\theta (E_i') = -1$ for all $i \in \{1, \dots, n-1\}$, $\theta (v_i') = 1$ for all $i \in \{1, \dots, n-1\}$, and $\theta (z) = -1$.
Let $i, j \in \{1, \dots, n-1\}$ such that $|i-j|=1$.
Then $\theta(E_i' v_j' E_i') = 1$ and $\theta (E_i') = -1$, hence the relation $E_i' v_j' E_i' = E_i'$ does not hold in $\VTL_n'$.
So, $\varphi$ is not injective.
\end{proof}

Let $n \ge 1$.
Recall that $V_n= \{0, 1\} \times \{1, \dots, n\}$ is ordered by $(0,1)< (0,2) < \cdots < (0,n) < (1,n) < \cdots <(1,2) < (1,1)$.
Let $E$ be a flat virtual $n$-tangle. 
Let $\gamma_1 = \{x_1, y_1\}, \gamma_2=\{x_2, y_2\}\in E$ such that $x_1 < y_1$, $x_2 < y_2$, and $x_1 < x_2$.
We say that $\gamma_1$ \emph{crosses} $\gamma_2$ if $x_1 <x_2 <y_1 <y_2$.
We say that $E$ is \emph{non-crossing} if there are no two elements in $E$ that cross. 
Equivalently, a flat $n$-tangle is non-crossing if and only if it has a graphical representation with $n$ disjoint arcs. 
We denote by $\EE_n^0$ the set of non-crossing flat virtual $n$-tangles. 

Let $n \ge 2$.
Recall that the \emph{Temperley--Lieb algebra} $\TL_n$ is the algebra over $R_0^f$ defined by the presentation with generators $E_1, \dots, E_{n-1}$ and relations 
\begin{gather*}
E_i^2 = z E_i \text{ for } 1 \le i \le n-1\,,\
E_i E_j = E_j E_i \text{ for }|i-j| \ge 2\,,\\
E_i E_j E_i = E_i \text{ for } |i-j|=1\,.
\end{gather*} 
The following is proved in Kauffman \cite{Kauff4}.

\begin{prop}[Kauffman \cite{Kauff4}]\label{prop2_4}
Let $n \ge 2$.
The homomorphism $\TL_n \to \VTL_n$ which sends $E_i$ to $E_i$ for all $i \in \{1, \dots, n-1\}$ is injective and its image is the $R_0^f$-submodule of $\VTL_n$ freely generated by $\EE_n^0$.
\end{prop}

Recall that $R^f = \Z [A^{\pm 1}]$ denotes the algebra of Laurent polynomials in the variable $A$, and that $R_0^f = \Z [z]$ is a subalgebra of $R^f$ via the identification $z = -A^2 -A^{-2}$.
For each $n \ge 1$ we set $\VTL_n (R^f) = R^f \otimes \VTL_n$.
This is a $R^f$-algebra and it is a free $R^f$-module freely generated by $\EE_n$.

\begin{thm}\label{thm2_5}
Let $n \ge 1$.
There exists a homomorphism $\rho_n^f: R^f [\VB_n] \to \VTL_n (R^f)$ which sends $\sigma_i$ to $-A^{-2}\,1-A^{-4}\,E_i$ and $\tau_i$ to $v_i$ for all $i \in \{1, \dots, n-1\}$.
\end{thm}

\begin{proof}
We set $S_i = -A^{-2}\,1 -A^{-4}E_i$.
We have 
\begin{gather*}
(- A^{-2} \, 1 - A^{-4} E_i)(- A^2 \, 1 - A^4 E_i) =
1 + (A^2 + A^{-2}) E_i + E_i^2 =\\
1 + (A^2 + A^{-2}) E_i + (- A^2 - A^{-2}) E_i=
1\,.
\end{gather*}
So, $S_i$ is invertible and its inverse is $- A^2 \, 1 - A^4 E_i$.
The element $v_i$ is also invertible since $v_i^2=1$.
It remains to verify that the following relations hold.
\begin{gather*}
v_i^2 = 1 \,, \quad \text{for } 1 \le i \le n-1\,,\\
S_i S_j = S_j S_i \,,\ v_i v_j = v_j v_i\,,\ v_i S_j = S_j v_i\,,\quad\text{for } |i-j|\ge2\,,\\
S_i S_j S_i = S_j S_i S_j\,,\ v_i v_j v_i = v_j v_i v_j\,,\ v_i v_j S_i = S_j v_i v_j\,,\quad\text{for } |i-j| = 1\,.
\end{gather*}
The only of these relations which is not trivial is $S_i S_j S_i = S_j S_i S_j$ for $|i-j| = 1$.
Suppose $|i-j|=1$.
Then
\begin{gather*}
S_i S_j S_i =
(- A^{-2} \, 1 - A^{-4} E_i) (- A^{-2} \, 1 - A^{-4} E_j) (- A^{-2} \, 1 - A^{-4} E_i) = \\
- A^{-6} \, 1 - A^{-8} E_i - A^{-8} E_j - A^{-10} E_i E_j - A^{-8} E_i - A^{-10} E_i^2 - A^{-10} E_j E_i - A^{-12} E_i E_j E_i =\\
- A^{-6}\,1 - 2 A^{-8} E_i - A^{-8} E_j - A^{-10} E_i E_j - A^{-10} (- A^2 - A^{-2}) E_i - A^{-10} E_j E_i - A^{-12} E_i =\\
- A^{-6} \,1 - A^{-8} E_i - A^{-8} E_j - A^{-10} E_i E_j - A^{-10} E_j E_i\,.
\end{gather*}
By symmetry we also have $S_j S_i S_j = - A^{-6} \, 1 - A^{-8} E_i - A^{-8} E_j - A^{-10} E_i E_j - A^{-10} E_j E_i$, hence $S_i S_j S_i = S_j S_i S_j$.
\end{proof}

\begin{rem}
\begin{itemize}
\item[(1)]
Recall that $B_n$ denotes the braid group on $n$ strands and $\TL_n$ denotes the $n$-th Temperley--Lieb algebra.
Then $\rho_n^f (\beta) \in \TL_n(R^f)$ for all $\beta \in B_n$, where $\TL_n (R^f) = R^f \otimes \TL_n \subset \VTL_n (R^f)$.
\item[(2)]
The sequence of homomorphisms $\{ \rho_n^f: R^f [\VB_n] \to \VTL_n(R^f)\}_{n=1}^\infty$ is compatible with the tower of algebras $\{\VTL_n (R^f)\}_{n=1}^\infty$.
\item[(3)]
Setting $\rho_n^f (\sigma_i) = -A^{-2} \, 1 - A^{-4} E_i$ instead of $\rho_n^f (\sigma_i) = A \,1 + A^{-1} E_i$, as an informed reader may expect, allows to include in $\rho_n^f$ the corrective with the writhe and to define directly the $f$-polynomial without passing through the Kauffman bracket. \end{itemize}
\end{rem}

Let $E$ be a flat virtual $n$-tangle.
By connecting with an arc the point $(0,i)$ with the point $(1,i)$ for all $i \in \{1, \dots, n\}$ in a diagram of $E$ we obtain a family of closed curves that we call the \emph{closure of the diagram} of $E$. 
We denote by $t_n(E)$ the number of closed curves in this family, and we set $T_n'^f(E) = z^{t_n(E)} = (-A^2 - A^{-2})^{t_n(E)}$.
Then we define $T_n'^f: \VTL_n (R^f) \to R^f$ by extending linearly the map $T_n'^f: \EE_n \to R^f$. 

\begin{expl}
In Figure \ref{fig2_5} is illustrated the closure of the flat virtual tangle $E$ of Figure \ref{fig2_2}.
In this case we have $t_n(E)=1$, and therefore $T_n'^f (E) = z = -A^2-A^{-2}$.
\end{expl}

\begin{figure}[ht!]
\begin{center}
\includegraphics[width=4cm]{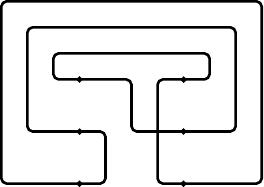}
\caption{Closure of a flat virtual tangle}\label{fig2_5}
\end{center}
\end{figure}

\begin{thm}\label{thm2_6}
The sequence $\{T_n'^f : \VTL_n(R^f) \to R^f\}_{n=1}^\infty$ is a Markov trace.
\end{thm}

\begin{proof}
For each $n\ge 2$ and each $i\in\{1,\dots,n-1\}$ we set $S_i=-A^{-2}\,1-A^{-4}E_i$.
We have to show that the following equalities hold.
\begin{itemize}
\item[(1)]
$T_n'^f(xy)=T_n'^f(yx)$ for all $n\ge1$ and all $x,y\in\VTL_n(R^f)$.
\item[(2)]
$T_n'^f(x)=T_{n+1}'^f(xS_n)=T_{n+1}'^f(xS_n^{-1})=T_{n+1}'^f(xv_n)$ for all $n\ge1$ and all $x\in\VTL_n(R^f)$.
\item[(3)]
$T_n'^f(x)=T_{n+1}'^f(xS_n^{-1}v_{n-1}S_n)$ for all $n\ge2$ and all $x\in\VTL_n(R^f)$.
\item[(4)]
$T_n'^f(x)=T_{n+1}'^f(xv_nv_{n-1}S_{n-1}v_nS_{n-1}^{-1}v_{n-1}v_n)$ for all $n\ge2$ and all $x\in\VTL_n(R^f)$.
\end{itemize}

{\it Proof of (1).}
We can assume that $x=E$ and $y=E'$ are flat virtual $n$-tangles.
By concatenating $E$ and $E'$ and connecting with an arc the point $(0,i)$ of $E$ to the point $(1,i)$ of $E'$ for all $i\in\{1,\dots,n\}$ we obtain a family of closed curves. 
If $m$ is the number of closed curves in this family, then $T_n'^f(EE')=z^m$.
We observe that, by concatenating $E'$ and $E$ and connecting with an arc the point $(0,i)$ of $E'$ to the point $(1,i)$ of $E$ for all $i\in\{1,\dots,n\}$, we also obtain a family of $m$ closed curves, hence $T_n'^f(E'E)=z^m=T_n'^f(EE')$.

{\it Proof of (2).}
We can assume that $x=E$ is a flat virtual $n$-tangle. 
We see in Figure \ref{fig2_6} that the following equalities hold 
\[
T_{n+1}'^f (E) =z \, T_n'^f (E)\,,\ T_{n+1}'^f (E E_n) = T_n'^f (E) \,, \ T_{n+1}'^f (E v_n) = T_n'^f (E)\,.
\]
Recall that $S_n^{-1} = -A^2 \,1 - A^4 E_n$ (see the proof of Theorem \ref{thm2_5}).
We saw in Figure \ref{fig2_6} that $T_{n+1}'^f (E v_n) = T_n'^f (E)$.
On the other hand,  
\begin{gather*}
T_{n+1}'^f (E S_n) =
- A^{-2} T_{n+1}'^f (E) - A^{-4} T_{n+1}'^f (E E_n) =
(- A^{-2} z - A^{-4}) T_n'^f (E) =
T_n'^f (E)\,,\\
T_{n+1}'^f (E S_n^{-1}) =
-A^2 T_{n+1}'^f (E) - A^4 T_{n+1}'^f (E E_n)=
(-A^2 z - A^4) T_n'^f (E)=
T_n'^f (E)\,.
\end{gather*}

\begin{figure}[ht!]
\begin{center}
\includegraphics[width=13.2cm]{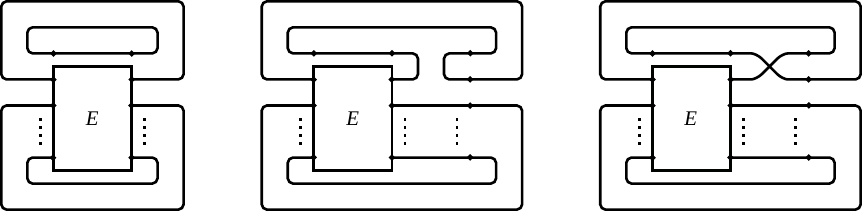}
\caption{$n+1$-closures of $E$, $EE_n$ and $Ev_n$}\label{fig2_6}
\end{center}
\end{figure}

{\it Proof of (3).}
We can again assume that $x=E$ is a flat virtual $n$-tangle.
We have
\begin{gather*}
E S_n^{-1} v_{n-1} S_n =
E (-A^2\,1-A^4E_n)v_{n-1}(-A^{-2}\,1-A^{-4}E_n)=\\
Ev_{n-1}+A^{2}EE_nv_{n-1}+A^{-2}Ev_{n-1}E_n+EE_nv_{n-1}E_n=\\
Ev_{n-1}+A^{2}EE_nv_{n-1}+A^{-2}Ev_{n-1}E_n+EE_n\,.
\end{gather*}
Hence, by the above
\begin{gather*}
T_{n+1}'^f (E S_n^{-1} v_{n-1} S_n) =\\
T_{n+1}'^f (E v_{n-1}) + A^{-2} T_{n+1}'^f (E v_{n-1} E_n) + A^{2} T_{n+1}'^f (E E_n v_{n-1}) + T_{n+1}'^f (E E_n) =\\
z  T_n'^f (E v_{n-1}) + A^{-2} T_n'^f (E v_{n-1}) + A^{2} T_{n+1}'^f (v_{n-1} E E_n) + T_n'^f(E) =\\
(-A^{2}-A^{-2}) T_n'^f (E v_{n-1}) + A^{-2} T_n'^f (E v_{n-1}) + A^{2} T_n'^f (v_{n-1} E) + T_n'^f(E) =\\
(-A^{2}-A^{-2}) T_n'^f (E v_{n-1}) + A^{-2} T_n'^f (E v_{n-1}) + A^{2} T_n'^f (E v_{n-1}) + T_n'^f(E) =
T_n'^f(E)\,.
\end{gather*}

{\it Proof of (4).}
Again, we can assume that $x=E$ is a flat virtual $n$-tangle. 
We have
\begin{gather*}
E v_n v_{n-1} S_{n-1} v_n S_{n-1}^{-1} v_{n-1} v_n =
E v_n v_{n-1} (-A^{-2}\, 1 - A^{-4} E_{n-1}) v_n (-A^{2}\, 1 - A^{4} E_{n-1}) v_{n-1} v_n =\\
( E v_n v_{n-1} v_n v_{n-1} v_n ) +
A^{2} (E v_n v_{n-1} v_n E_{n-1} v_{n-1} v_n ) +\\
A^{-2} (E v_n v_{n-1} E_{n-1} v_n v_{n-1} v_n ) +
(E v_n v_{n-1} E_{n-1} v_n E_{n-1} v_{n-1} v_n ) =\\
( E v_{n-1} v_n v_{n-1} v_{n-1} v_n ) +
A^{2} (E v_n v_{n-1} v_n E_{n-1} v_n ) +\\
A^{-2} (E v_n E_{n-1} v_n v_{n-1} v_n ) +
(E v_n  E_{n-1} v_n E_{n-1} v_n ) =\\
( E v_{n-1} ) +
A^{2} (E v_n v_{n-1} v_{n-1} E_n v_{n-1} ) +
A^{-2} (E v_{n-1} E_n v_{n-1} v_{n-1} v_n ) +
(E v_n  E_{n-1} v_n ) =\\
( E v_{n-1} ) +
A^{2} (E E_n v_{n-1} ) +
A^{-2} (E v_{n-1} E_n ) +
(E v_{n-1}  E_n v_{n-1} )\,.
\end{gather*}
Hence, by the above
\begin{gather*}
T_{n+1}'^f (E v_n v_{n-1} S_{n-1} v_n S_{n-1}^{-1} v_{n-1} v_n) =\\
T_{n+1}'^f ( E v_{n-1} ) +
A^{2} T_{n+1}'^f (E E_n v_{n-1} ) +
A^{-2} T_{n+1}'^f (E v_{n-1} E_n ) +
T_{n+1}'^f (E v_{n-1}  E_n v_{n-1} ) =\\
z T_n'^f ( E v_{n-1} ) +
A^{2} T_{n+1}'^f (v_{n-1} E E_n ) +
A^{-2} T_n'^f (E v_{n-1} ) +
T_{n+1}'^f (v_{n-1} E v_{n-1}  E_n ) =\\
(-A^{2} - A^{-2}) T_n'^f ( E v_{n-1} ) +
A^{2} T_n'^f (v_{n-1} E ) +
A^{-2} T_n'^f (E v_{n-1} ) +
T_n'^f (v_{n-1} E v_{n-1} ) =\\
(-A^{2} - A^{-2}) T_n'^f ( E v_{n-1} ) +
A^{2} T_n'^f ( E v_{n-1} ) +
A^{-2} T_n'^f ( E v_{n-1} ) +
T_n'^f ( E v_{n-1} v_{n-1} ) =
T_n'^f (E )\,.\qquad \proved
\end{gather*}
\end{proof}

\begin{corl}\label{corl2_7}
For each $n \ge 1$ we set $T_n^f=T_n'^f \circ \rho_n^f: R^f [\VB_n] \to R^f$.
Then $\{T_n^f: R^f [\VB_n] \to R^f\}_{n=1}^\infty$ is a Markov trace.
\end{corl}

Recall that $\VV\LL$ denotes the set of virtual links.
To complete the study of this section it remains to prove the following. 

\begin{thm}\label{thm2_8}
Let $I^f: \VV\LL\to R^f$ be the invariant defined from the Markov trace of Corollary~\ref{corl2_7}.
Then $I^f$ coincides with the $f$-polynomial.
\end{thm}

\begin{proof}
Let $\beta$ be a virtual braid on $n$ strands and let $\hat\beta$ be its closure. 
Observe that the relation $\langle L \rangle = A\,\langle L_1 \rangle + A^{-1} \, \langle L_2 \rangle$ in the definition of the Kauffman bracket corresponds in terms of closed virtual braids to replacing each $\sigma_i$ with $A \, 1 + A^{-1} E_i$ and each $\sigma_i^{-1}$ with $A^{-1} \, 1 + A E_i$.
Once we have replaced each $\sigma_i$ with $A \, 1 + A^{-1} E_i$ and each $\sigma_i^{-1}$ with $A^{-1} \,1 +A \, E_i$, we get a linear combination $\sum_{i=1}^\ell a_i E^{(i)}$, where $E^{(i)} \in \EE_n$ and $a_i \in R^f$.
For each $i \in \{1, \dots, \ell\}$ we denote by $m_i = t_n (E^{(i)})$ the number of closed curves in the closure of $E^{(i)}$.
We see that
\[
\langle \hat \beta \rangle = \sum_{i=1}^\ell a_i z^{m_i}\,.
\]
Recall that $w :\VV \LL  \to \Z$ denotes the writhe.
Let $\omega: \VB_n \to \Z$ be the homomorphism which sends $\sigma_i$ to $1$ and $\tau_i$ to $0$ for all $i \in \{1, \dots, n-1\}$.
Then $w (\hat \beta) = \omega (\beta)$ and therefore $f (\hat \beta) = (-A^3)^{-\omega (\beta)} \langle \hat \beta \rangle$.
So, in the above procedure, if we replace each $\sigma_i$ with $(-A^3)^{-1} (A \,1 + A^{-1} E_i)= -A^{-2} \,1 - A^{-4} E_i$ and each $\sigma_i^{-1}$ with $(-A^3) (A^{-1} \,1 + A E_i) = -A^2 \,1 - A^4 E_i$, then we get directly $f (\hat \beta)$.
It is clear that this procedure also leads to $T_n^f (\beta)$.
\end{proof}


\section{Arrow Temperley--Lieb algebras and arrow polynomial}\label{sec3}

Throughout the section we consider the infinite families of variables $\ZZ=\{z_k\}_{k=0}^\infty$ and $\ZZ^*=\{z_k\}_{k=1}^\infty = \ZZ\setminus\{z_0\}$, and we consider the algebra $R_0^a=\Z[\ZZ]$ of polynomials in the variables in $\ZZ$, and the algebra $R^a=\Z[A^{\pm 1},\ZZ^*]$ of Laurent polynomials in the variable $A$ and standard polynomials in the variables in $\ZZ^*$.
We also assume that the algebra $R_0^a$ is embedded into $R^a$ via the identification $z_0=-A^2-A^{-2}$.
Following the same strategy as in Section \ref{sec2}, we start by recalling the definition of the arrow polynomial, so that the reader will understand easier the constructions that will follow after.

Let $S_1,\dots,S_\ell$ be a collection of $\ell$ circles smoothly immersed in the plane and having only a finite number of normal double crossings.
We assume that each circle $S_i$ has an even number $m_i$ of marked points outside the crossings that we call \emph{cusps}.
We assume also that each segment between two successive cusps is oriented so that the orientations of the two segments adjacent to a given cusp are opposite. 
So, each cusp is either a \emph{sink} or a \emph{source}, according to the orientations of the segments adjacent to it (see Figure \ref{fig3_1}).
If $m_i=0$, then $S_i$ is assumed to have a (unique) orientation. 
In addition, each cusp has a privileged side that we indicate with a small segment like in Figure \ref{fig3_1}.
Finally, as for the virtual link diagrams, we assign a value ``positive'', ``negative'', or ``virtual'' to each crossing, that we indicate in its graphical representation as in Figure \ref{fig1_1}.
Such a figure is called an \emph{arrow virtual link diagram} with $\ell$ components. 
Note that the virtual link diagrams are the arrow virtual link diagrams with no cusps.

\begin{figure}[ht!]
\begin{center}
\includegraphics[width=4cm]{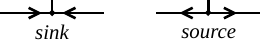}
\caption{Sink and source in an arrow virtual link diagram}\label{fig3_1}
\end{center}
\end{figure}

\begin{expl}
Figure \ref{fig3_2} shows an arrow virtual link diagram with two components. 
One component has two cusps and the other has no cusp.
\end{expl}

\begin{figure}[ht!]
\begin{center}
\includegraphics[width=2.6cm]{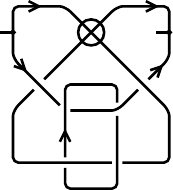}
\caption{Arrow virtual link diagram}\label{fig3_2}
\end{center}
\end{figure}

Let $L$ be an arrow virtual link diagram with only virtual crossings.
Let $S$ be a component of $L$.
If $S$ has two consecutive cusps $p$ and $q$ having the same privileged side, then we remove the two cusps and orient the new arc with the same orientation as that of the arc adjacent to $p$ different from $[p,q]$.
In the particular case where $p$ and $q$ are the only cusps of $S$, then we can choose any of the orientations of $S$.
This operation is called a \emph{reduction} of $S$ and is illustrated in Figure \ref{fig3_3}.
We apply such a reduction as many times as needed to get an \emph{irreducible component}, $S'$.
If $2c$ is the number of cusp of $S'$, then $c$ is called the \emph{number of zigzags} of $S$ and is denoted by $\zeta (S)=c$.
If $S_1,\dots, S_\ell$ are the components of $L$, then we set
\[
\langle\!\langle L\rangle\!\rangle=\prod_{i=1}^\ell z_{\zeta(S_i)}\,.
\]
This is a monomial of $R_0^a$.

\begin{figure}[ht!]
\begin{center}
\includegraphics[width=6.8cm]{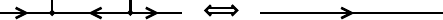}
\caption{Reduction}\label{fig3_3}
\end{center}
\end{figure}

We define the \emph{arrow Kauffman bracket} $\langle\!\langle L\rangle\!\rangle\in R^a$ of any arrow virtual link diagram $L$ as follows.
If $L$ has only virtual crossings, then $\langle\!\langle L\rangle\!\rangle$ is the monomial $\prod_{i=1}^\ell z_{\zeta (S_i)}$ defined above. 
Suppose that $L$ has at least one non-virtual crossing at a point $p$.
If the crossing is positive, then we set $\langle\!\langle L\rangle\!\rangle=A\langle\!\langle L_1\rangle\!\rangle+A^{-1}\langle\!\langle L_2\rangle\!\rangle$, and, if the crossing is negative, then we set $\langle\!\langle L\rangle\!\rangle=A^{-1}\langle\!\langle L_1\rangle\!\rangle+A\langle\!\langle L_2\rangle\!\rangle$, where $L_1$ and $L_2$ are identical to $L$ except in a small neighborhood of $p$ where there are as shown in Figure \ref{fig3_4}.
As for the virtual link diagrams, the \emph{writhe} of an arrow virtual link diagram $L$, denoted $w(L)$, is the number of positive crossings menus the number of negative crossings. 
Then the \emph{arrow polynomial} of an arrow virtual link diagram $L$ is defined by $\overrightarrow{f} (L)=(-A^3)^{-w(L)}\langle\!\langle L\rangle\!\rangle$.

\begin{figure}[ht!]
\begin{center}
\includegraphics[width=8.4cm]{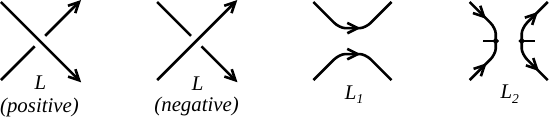}
\caption{Relation in the arrow Kauffman bracket}\label{fig3_4}
\end{center}
\end{figure}

\begin{thm}[Miyazawa \cite{Miyaz1}, Dye--Kauffman \cite{DyeKau1}]\label{thm3_1}
\begin{itemize}
\item[(1)]
If two virtual link diagrams $L$ and $L'$ are equivalent, then $\overrightarrow{f} (L)=\overrightarrow{f} (L')$.
\item[(2)]
If $L$ is a diagram of a classical link, then $\overrightarrow{f} (L) = f(L) \in R^f =\Z [A^{\pm 1}]$.
\end{itemize}
\end{thm}

\begin{rem}
There is a notion of ``equivalence'' between arrow virtual link diagrams and Theorem \ref{thm3_1} holds in this framework (see Miyazawa \cite{Miyaz1}), but the topic of the present paper are the virtual links, hence we state the theorem only for virtual link diagrams.
\end{rem}

The \emph{arrow polynomial} of a virtual link $L$, denoted $\overrightarrow{f} (L)$, is defined to be the arrow polynomial of any of its diagrams. 
This is a well-defined invariant thanks to Theorem \ref{thm3_1}.

Our aim now is to construct a Markov trace whose associated invariant is the arrow polynomial. 
We proceed with the same strategy as in Section \ref{sec2} for the $f$-polynomial: we pass through a tower of algebras, $\{\ATL_n\}_{n=1}^\infty$, that we will call arrow Temperley--Lieb algebras. 
We will also give a new proof/interpretation of Theorem \ref{thm3_1}\,(2) in terms of Markov traces.

For the remainder of the section we need a more combinatorial definition of the multiplication in $\VTL_n$.
Recall that $V_n = \{ 0, 1 \} \times \{ 1, \dots, n\}$ is ordered by $(0,1) < (0,2) < \cdots < (0,n) < (1,n) < \cdots < (1,2) < (1,1)$.
Let $E, E' \in \EE_n$.
An \emph{arc} of length $\ell$ in $E \sqcup E'$ is a $\ell$-tuple $\hat \alpha = (\alpha_1, \dots, \alpha_\ell)$ in $E \sqcup E'$, where $\alpha_i = \{(a_i, b_{i-1}),(c_i, b_i) \}$, satisfying the following properties.
\begin{itemize}
\item[(a)]
If $\alpha_i \in E$ and $i<\ell$, then $c_i=1$, $\alpha_{i+1} \in E'$ and $a_{i+1}=0$.
\item[(b)]
If $\alpha_i \in E'$ and $i<\ell$, then $c_i=0$, $\alpha_{i+1} \in E$ and $a_{i+1}=1$.
\item[(c)]
$a_0=0$ if $\alpha_1 \in E$, $a_0 = 1$ if $\alpha_1 \in E'$, $c_\ell=0$ if $\alpha_\ell \in E$, $c_\ell = 1$ if $\alpha_\ell \in E'$, and $(a_1,b_0) < (c_\ell,b_\ell)$.
\end{itemize}
The \emph{boundary} of $\hat\alpha$ is $\partial\hat\alpha=\{(a_1,b_0),(c_\ell,b_\ell) \}$.
There are $n$ arcs in $E\sqcup E'$ and their boundaries form a flat virtual $n$-tangle, denoted $E*E'$.

Let $E, E' \in \EE_n$.
A \emph{cycle} of length $2p \ge 2$ in $E \sqcup E'$ is a $2p$-tuple $\hat \gamma = (\gamma_1, \dots, \gamma_{2p})$ in $E \sqcup E'$, where $\gamma_i = \{ (a_i, b_{i-1}), (c_i, b_i) \}$, satisfying the following properties.
\begin{itemize}
\item[(a)]
$\gamma_i \in E$ and $a_i= c_i = 1$, if $i$ is odd.
\item[(b)]
$\gamma_i \in E'$ and $a_i = c_i = 0$, if $i$ is even.
\item[(c)]
$b_0 = b_{2p} < b_i$ for all $i \in \{1, \dots, 2p-1\}$.
\end{itemize}
Let $m$ be the number of cycles in $E \sqcup E'$.
Then $E\,E'=z^m(E*E')$.

We can now define our algebra $\ATL_n$.
Let $n \ge 1$.
An \emph{arrow flat $n$-tangle} is a flat virtual $n$-tangle $E$ endowed with a labeling $f:E\to\Z$ 
such that, for $\alpha = \{ (a,b), (c,d) \} \in E$, $f(\alpha)$ is odd if $a=c$, and $f (\alpha)$ is even if $a \neq c$.
Recall that $\ZZ=\{z_k\}_{k=0}^\infty$ and $R_0^a=\Z[\ZZ]$.
We denote by $\FF_n$ the set of arrow flat $n$-tangles and by $\ATL_n$ the free $R_0^a$-module freely generated by $\FF_n$.

\begin{interpret}
Instead of labeling the arcs we could endow each arc with marked points (cusps) and each cusp with a privileged side that we indicate with a small segment, like for arrow virtual link diagrams, so that two consecutive cusps have different privileged sides. 
The number of cusps on an arc $\alpha$ would be equal to $|f(\alpha)|$.
Consider the order of $V_n$ defined above. 
If we travel on the arc from its smallest extremity to its largest one, we set $f(\alpha)>0$ if the privileged side of the first encountered cusp is on the left hand side, and we set $f(\alpha)<0$ otherwise.
An arrow flat tangle and its version with cusps are illustrated in Figure \ref{fig3_5}.
\end{interpret}

\begin{figure}[ht!]
\begin{center}
\includegraphics[width=4.5cm]{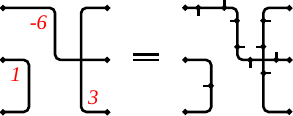}
\caption{Arrow flat tangle}\label{fig3_5}
\end{center}
\end{figure}

We now define the multiplication in $\ATL_n$.
Let $F=(E,f)$ and $F'=(E',f')$ be two arrow flat $n$-tangles. 
To simplify our notation we set $f^*(\alpha)=f(\alpha)$ if $\alpha\in E$ and $f^*(\alpha)=f'(\alpha)$ if $\alpha\in E'$.
The \emph{parity} of an element $\alpha = \{ (a,b), (c,d) \} \in E \sqcup E'$ is $\varpi (\alpha) = -1$ if $a=c$, and $\varpi (\alpha) = 1$ if $a \neq c$.
Let $\hat\alpha=(\alpha_1, \dots, \alpha_\ell )$ be an arc of $E \sqcup E'$.
Let $i \in \{1, \dots, \ell\}$.
As in the above definition of arc, we set $\alpha_i = \{ (a_i,b_{i-1}), (c_i, b_i)\}$ for all $i$.
We define the \emph{cumulated parity} of $\alpha_i$ relative to $\hat \alpha$ by $\varpi^c(\alpha_i) = \prod_{j=1}^{i-1} \varpi(\alpha_j)$ if $(a_i,b_{i-1}) < (c_i,b_i)$, and $\varpi^c(\alpha_i) = \prod_{j=1}^{i} \varpi(\alpha_j)$ if $(c_i,b_i) < (a_i, b_{i-1})$.
Then we set
\[
g(\partial\hat\alpha)=g(\hat\alpha)=\sum_{i=1}^\ell \varpi^c(\alpha_i)\, f^*(\alpha_i)\,.
\]
At this stage we have an arrow flat $n$-tangle $F*F'=(E*E',g)$.
Let $\hat\gamma=(\gamma_1, \dots, \gamma_\ell)$ be a cycle of $E\sqcup E'$.
Again, we write $\gamma_i = \{(a_i, b_{i-1}), (c_i, b_i)\}$ for all $i$.
As for an arc, we define the \emph{cumulated parity} of $\gamma_i$ relative to $\hat \gamma$ by $\varpi^c(\gamma_i) = \prod_{j=1}^{i-1} \varpi(\gamma_j)$ if $(a_i,b_{i-1}) < (c_i,b_i)$, and by $\varpi^c(\gamma_i) = \prod_{j=1}^{i} \varpi(\gamma_j)$ if $(c_i, b_i) < (a_i, b_{i-1})$. 
We set
\[
h(\hat\gamma)=\sum_{i=1}^\ell\varpi^c (\gamma_i)\, f^*(\gamma_i)\,.
\]
Observe that $h(\hat\gamma)$ is an even number.
The \emph{number of zigzags} of $\hat\gamma$ is defined by $\zeta(\hat\gamma)=\frac{|h(\hat\gamma)|}{2}$.
Let $\hat\gamma_1,\dots,\hat\gamma_m$ be the cycles of $E\sqcup E'$.  
Then the product of $F$ and $F'$ is 
\[
F\,F'=z_{\zeta(\hat\gamma_1)}\cdots z_{\zeta(\hat\gamma_m)}(F*F')\,.
\]
It is easily checked that $\ATL_n$ endowed with this multiplication is an (associative and unitary) algebra.
We call it the $n$-th \emph{arrow Temperley--Lieb algebra}.

\begin{expl}
On the left hand side of Figure \ref{fig3_6} are illustrated two arrow flat tangles $F$ and $F'$, and  $F * F'$ is illustrated on the right hand side.
Here we have a unique cycle in $E \sqcup E'$, $\hat\gamma_1$, and $h(\hat\gamma_1)=4$, hence $\zeta(\hat\gamma_1)=2$ and $F\,F'=z_2(F*F')$.
\end{expl}

\begin{figure}[ht!]
\begin{center}
\includegraphics[width=7.8cm]{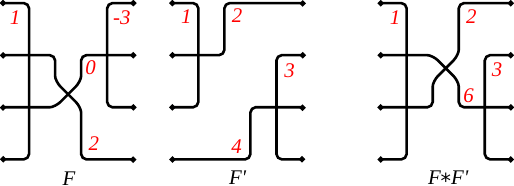}
\caption{Multiplication in $\ATL_n$}\label{fig3_6}
\end{center}
\end{figure}

\begin{rem}
Let $n \ge 1$. 
For $F=(E,f) \in \FF_n$ we define $F^\sharp =(E^\sharp, f^\sharp) \in \FF_{n+1}$ by setting $E^\sharp = E \cup \{ \{ (0,n+1), (1,n+1) \} \}$, $f^\sharp(\alpha) = f(\alpha)$ for all $\alpha \in E$, and $f^\sharp (\{ (0,n+1), (1,n+1) \}) = 0$.
Then the map $\FF_n \to \FF_{n+1}$, $F \mapsto F^\sharp$, is an embedding which induces an injective homomorphism $\ATL_n \hookrightarrow \ATL_{n+1}$.
So, we have a tower of algebras $\{\ATL_n\}_{n=1}^\infty$.
\end{rem}

\begin{prop}\label{prop3_2}
Let $n \ge 2$.
Then $\ATL_n$ has a presentation with generators
\[
F_1,\dots,F_{n-1},w_1,\dots,w_{n-1},t_1,\dots,t_n,t_1^{-1},\dots,t_n^{-1}\,,
\]
and relations
\begin{gather*}
t_it_i^{-1}=t_i^{-1}t_i=1\text{ for } 1\le i\le n\,,\quad
t_it_j=t_jt_i\text{ for }1\le i<j\le n\,,\\
w_i^2=1\text{ for }1\le i\le n-1\,,\quad
w_iw_j=w_jw_i\text{ for } |i-j|\ge 2\,,\\
w_iw_jw_i=w_jw_iw_j\text{ for }|i-j|=1\,,\quad
w_it_i=t_{i+1}w_i\text{ for }1\le i\le n-1\,,\\
w_it_{i+1}=t_iw_i\text{ for }1\le i\le n-1\,,\quad
w_it_j=t_jw_i\text{ for }j\neq i,i+1\,,\\
F_it_i^mF_i=z_{|m|}F_i\text{ for }1\le i \le n-1 \text{ and }m\in\Z\,,\quad
F_iw_i=F_it_i=F_it_{i+1}^{-1}\text{ for }1\le i\le n-1\,,\\
w_iF_i=t_i^{-1}F_i=t_{i+1}F_i\text{ for }1 \le i\le n-1\,,\quad
F_iF_j=F_jF_i\text{ for }|i-j|\ge2\,,\\
F_iw_j=w_jF_i\text{ for }|i-j|\ge2\,,\quad
F_it_j=t_jF_i\text{ for }j\neq i,i+1\,,\\
F_iw_jF_i=F_i\text{ for }|i-j|=1\,,\quad
w_iw_jF_i=F_jw_iw_j\text{ for }|i-j|=1\,.
\end{gather*}
\end{prop}

The generators $F_i$, $w_i$ and $t_j$ are illustrated in Figure \ref{fig3_7}.

\begin{figure}[ht!]
\begin{center}
\includegraphics[width=6.6cm]{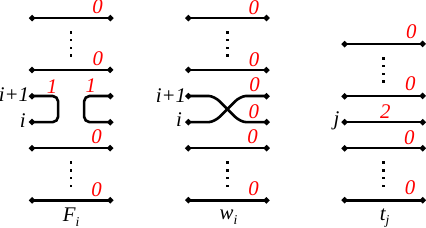}
\caption{Generators of $\ATL_n$}\label{fig3_7}
\end{center}
\end{figure}

\begin{proof}
Let $A_n$ be the $R_0^a$-algebra defined by a presentation with generators
\[
X_1,\dots,X_{n-1},y_1,\dots,y_{n-1},u_1,\dots,u_n,u_1^{-1},\dots,u_n^{-1}\,,
\]
and relations 
\begin{gather*}
u_iu_i^{-1}=u_i^{-1}u_i=1\text{ for } 1\le i\le n\,,\quad
u_iu_j=u_ju_i\text{ for }1\le i<j\le n\,,\\
y_i^2=1\text{ for }1\le i\le n-1\,,\quad
y_iy_j=y_jy_i\text{ for } |i-j|\ge 2\,,\\
y_iy_jy_i=y_jy_iy_j\text{ for }|i-j|=1\,,\quad
y_iu_i=u_{i+1}y_i\text{ for }1\le i\le n-1\,,\\
y_iu_{i+1}=u_iy_i\text{ for }1\le i\le n-1\,,\quad
y_iu_j=u_jy_i\text{ for }j\neq i,i+1\,,\\
X_iu_i^mX_i=z_{|m|}X_i\text{ for }1\le i \le n-1 \text{ and }m\in\Z\,,\quad
X_iy_i=X_iu_i=X_iu_{i+1}^{-1}\text{ for }1\le i\le n-1\,,\\
y_iX_i=u_i^{-1}X_i=u_{i+1}X_i\text{ for }1 \le i\le n-1\,,\quad
X_iX_j=X_jX_i\text{ for }|i-j|\ge2\,,\\
X_iy_j=y_jX_i\text{ for }|i-j|\ge2\,,\quad
X_iu_j=u_jX_i\text{ for }j\neq i,i+1\,,\\
X_iy_jX_i=X_i\text{ for }|i-j|=1\,,\quad
y_iy_jX_i=X_jy_iy_j\text{ for }|i-j|=1\,.
\end{gather*}
It is easily checked using diagrammatic calculation that there is a homomorphism $\varphi: A_n\to\ATL_n$ which sends $X_i$ to $F_i$ for $i\in\{1,\dots,n-1\}$, $y_i$ to $w_i$ for $i\in\{1,\dots,n-1\}$, and $u_i^{\pm1}$ to $t_i^{\pm1}$ for $i\in\{1,\dots,n\}$.
It remains to prove that $\varphi$ is an isomorphism.

{\it Claim 1.}
{\it The following equalities hold in $A_n$.
\begin{gather*}
X_iX_{i+1}y_iy_{i+1}=X_iu_{i+2}^{-1}=u_{i+2}^{-1}X_i\text{ for }1\le i\le n-2\,,\\
y_{i+1}y_iX_{i+1}X_i=u_{i+2}X_i=X_iu_{i+2}\text{ for }1\le i\le n-2\,,\\
X_iX_{i-1}y_iy_{i-1}=X_iu_{i-1}=u_{i-1}X_i\text{ for }2\le i\le n-1\,,\\
y_{i-1}y_iX_{i-1}X_i=u_{i-1}^{-1}X_i=X_iu_{i-1}^{-1}\text{ for }2 \le i\le n-1\,.
\end{gather*}}

{\it Proof of Claim 1.}
We prove the first equality. 
The other three can be proved in the same way. 
Let $i\in\{1,\dots,n-2\}$.
Then
\begin{gather*}
X_iX_{i+1}y_iy_{i+1}=
X_iy_iy_{i+1}X_i=
X_iu_{i+1}^{-1}y_{i+1}X_i=
X_iy_{i+1}u_{i+2}^{-1}X_i=
X_iy_{i+1}X_iu_{i+2}^{-1}=\\
X_iu_{i+2}^{-1}=
u_{i+2}^{-1}X_i\,.
\end{gather*}

{\it Claim 2.}
{\it Let $i,j\in\{1,\dots,n-1\}$ such that $|i-j|=1$.
Then $X_iX_jX_i=X_i$.}

{\it Proof of Claim 2.}
We suppose that $j=i+1$.
The case $j=i-1$ can be proved in the same way.
\begin{gather*}
X_iX_{i+1}X_i=
X_iX_{i+1}y_iy_{i+1}y_{i+1}y_iX_i=
X_iu_{i+2}^{-1}y_{i+1}u_{i+1}X_i=
X_iu_{i+2}^{-1}u_{i+2}y_{i+1}X_i=\\
X_iy_{i+1}X_i=
X_i\,.
\end{gather*}

{\it Claim 3.}
{\it Let $i\in\{1,\dots,n-3\}$.
Then $y_{i+1}y_{i+2}y_iy_{i+1}X_iX_{i+2}=X_iX_{i+2}$.}

{\it Proof of Claim 3.}
\begin{gather*}
y_{i+1}y_{i+2}y_iy_{i+1}X_iX_{i+2}=
y_{i+1}y_{i+2}y_iy_{i+1}X_iu_{i+2}u_{i+2}^{-1}X_{i+2}=\\
y_{i+1}y_{i+2}y_iy_{i+1}y_{i+1}y_iX_{i+1}X_iu_{i+2}^{-1}X_{i+2}=
y_{i+1}y_{i+2}X_{i+1}X_iu_{i+2}^{-1}X_{i+2}=\\
y_{i+1}y_{i+2}X_{i+1}u_{i+3}u_{i+3}^{-1}X_iu_{i+3}X_{i+2}=
y_{i+1}y_{i+2}y_{i+2}y_{i+1}X_{i+2}X_{i+1}X_iX_{i+2}=\\
X_{i+2}X_{i+1}X_{i+2}X_i=
X_{i+2}X_i=
X_iX_{i+2}\,.
\end{gather*}

Consider the action of the symmetric group $\SSS_n$ on $\Z^n$ by permutations of the coordinates, and set $G = \SSS_n \ltimes \Z^n$.
Let $\{e_1, \dots, e_n\}$ be the standard basis of $\Z^n$ and let $\{s_1, \dots, s_{n-1}\}$ be the standard set of generators of $\SSS_n$.
Recall that $s_i$ is the transposition $(i,i+1)$ for $i \in \{1,\dots,n-1\}$.
We use multiplicative notation for the operation in $\Z^n$ and we denote by $1_{\Z^n}$ its neutral element. 
Let $\UU (A_n)$ be the group of units of $A_n$.
We have a homomorphism $\iota: G \to \UU (A_n)$ which sends $s_i$ to $y_i$ for all $i \in \{1, \dots, n-1\}$ and $e_j$ to $u_j$ for all $j \in \{1, \dots, n\}$.

Let $n_0$ be the integer part of $\frac{n}{2}$.
For $p \in \{1, \dots, n_0\}$ we set $B_p = X_1 X_3 \cdots X_{2p-1}$, and for $p=0$ we set $B_p = B_0 = 1$. 
We denote by $U_{1,p}$ the subset of $G$ formed by the elements of the form $g=wh$ where $w \in \SSS_n$ satisfies
\begin{gather*}
w(1) < w(3) < \cdots < w(2p-1)\,,\quad
w(2i-1) < w(2i) \text{ for } 1 \le i \le p\,,\\
w(2p+1) < w(2p+2) < \cdots <w(n)\,,
\end{gather*}
and $h = e_1^{\nu_1} e_2^{\nu_2} \cdots e_n^{\nu_n} \in \Z^n$ satisfies 
\[
\nu_i = 0 \text{ for } i \in \{2, 4, \dots, 2p, 2p+1, 2p+2, \dots, n\}\,.
\]
On the other hand, we denote by $U_{2,p}$ the subset of $G$ formed by the elements of the form $g = w h$ where $w \in \SSS_n$ satisfies
\[
w(1) < w(3) < \cdots < w(2p-1)\,,\quad
w(2i-1) < w(2i) \text{ for } 1 \le i \le p\,,
\]
and $h = e_1^{\nu_1} e_2^{\nu_2} \cdots e_n^{\nu_n} \in \Z^n$ satisfies
\[
\nu_i = 0 \text{ for } i \in \{2, 4, \dots, 2p\}\,.
\]
Then we set 
\[
\BB_p = \{ \iota(g_1) \, B_p \, \iota(g_2^{-1}) \mid g_1 \in U_{1,p}\,,\ g_2 \in U_{2,p}\}\,,
\]
for $0 \le p \le n_0$, and  
\[
\BB = \bigcup_{p=0}^{n_0} \BB_p\,.
\]

Let $F = (E,f) \in \FF_n$.
We can write $E$ in the form $E = \{\alpha_1, \dots, \alpha_p, \alpha_1', \dots, \alpha_p', \beta_1, \dots, \beta_q\}$, where
\begin{itemize}
\item
each $\alpha_i$ is of the form $\alpha_i = \{(0,a_i), (0,b_i)\}$, with $a_1 < a_2 < \cdots <a_p$, and $a_i < b_i$ for all $i \in \{1,\dots,p\}$;
\item
each $\alpha_i'$ is of the form $\alpha_i' = \{(1,a_i'), (1,b_i')\}$, with $a_1' < a_2' < \cdots < a_p'$, and $a_i' < b_i'$ for all $i \in \{1, \dots, p\}$;
\item
each $\beta_j$ is of the form $\beta_j = \{(0,c_j), (1,d_j)\}$, with $c_1 < c_2 < \cdots < c_q$, and $2p+q = n$.
\end{itemize}
We define $w_1 \in \SSS_n$ by $w_1 (2i-1) = a_i$ and $w_1 (2i) = b_i$ for all $i \in \{1, \dots, p\}$, and $w_1 (2p+j) = c_j$ for all $j \in \{1, \dots, q\}$.
We define $w_2 \in \SSS_n$ by $w_2 (2i-1) = a_i'$ and $w_2 (2i) = b_i'$ for all $i \in \{1, \dots, p\}$, and $w_2 (2p+j) = d_j$ for all $j \in \{1, \dots, q\}$.
We define $h_1 = e_1^{\nu_{1,1}} e_2^{\nu_{1,2}} \cdots e_n^{\nu_{1,n}} \in \Z^n$ by $\nu_{1,2i-1} = \frac{f(\alpha_i) - 1}{2}$ and $\nu_{1,2i}=0$ for $i \in \{1, \dots, p\}$, and $\nu_{1,2p+j} = 0$ for  $j \in \{1, \dots, q\}$.
We define $h_2 = e_1^{\nu_{2,1}} e_2^{\nu_{2,2}} \cdots e_n^{\nu_{2,n}} \in \Z^n$ by $\nu_{2,2i-1} = \frac{f(\alpha_i')-1}{2}$ and $\nu_{2,2i}=0$ for $i \in \{1, \dots, p\}$, and $\nu_{2,2p+j} = -\frac{f(\beta_j)}{2}$ for $j \in \{1, \dots, q\}$.
We set $g_1 = w_1 h_1$ and $g_2 = w_2 h_2$.
Then $g_1 \in U_{1,p}$, $g_2 \in U_{2,p}$, and $F = \varphi \big( \iota (g_1) \, B_p \, \iota (g_2^{-1}) \big)$.
Moreover, such an expression is unique, and $\varphi (Y) \in \FF_n$ for all $Y \in \BB$.
So, $\varphi$ restricts to a bijection from $\BB$ to $\FF_n$.

So, in order to prove Proposition \ref{prop3_2}, it suffices to show that $\BB$ spans $A_n$ as a $R_0^a$-module.
Let $\MM$ be the submonoid of $A_n$ generated by $X_1, \dots, X_{n-1}, y_1, \dots, y_{n-1}, u_1^{\pm 1}, \dots, u_n^{\pm 1}$, that is, the set of finite products of elements in $\{X_1, \dots, X_{n-1}, y_1, \dots, y_{n-1}, u_1^{\pm 1}, \dots, u_n^{\pm 1}\}$.
By definition $\MM$ spans $A_n$ as a $R_0^a$-module, hence we only need to show that $\MM$ is contained in the $R_0^a$-submodule $\Span_{R_0^a} (\BB)$ of $A_n$ spanned by $\BB$.

{\it Claim 4.}
{\it Let $g_1, g_2 \in G$ and $p \in \{0, 1, \dots, n_0\}$.
Then $\iota (g_1) \, B_p \, \iota(g_2^{-1}) \in \BB$.}

{\it Proof of Claim 4.}
We write $g_1 = w_1 h_1$ and $g_2 = w_2 h_2$ with $w_1, w_2 \in \SSS_n$ and $h_1, h_2 \in \Z^n$.
Let $i \in \{1, \dots, p\}$ such that $w_1 (2i-1) = b_i > w_1 (2i) = a_i$.
By using the relation $y_i X_i = u_i^{-1} X_i$, we can replace $w_1$ with $w_1 s_i$ and $h_1$ with $s_i (h_1e_i) s_i \in \Z^n$, and then $w_1 (2i-1) = a_i <w_1 (2i) = b_i$.
So, we can assume that $w_1 (2i-1) < w_1 (2i)$ for all $i \in \{1, \dots, p\}$.
Let $i \in \{1, \dots, p-1\}$ such that $w_1 (2i-1) = a_{i+1} > w_1 (2i+1) = a_i$.
By Claim 3 we can replace $w_1$ with $w_1 s_{2i} s_{2i-1} s_{2i+1} s_{2i}$ and $h_1$ with  $(s_{2i} s_{2i+1} s_{2i-1} s_{2i}) h_1 (s_{2i} s_{2i-1} s_{2i+1} s_{2i}) \in \Z^n$.
Then we have $w_1 (2i-1) = a_i < w_1 (2i+1) = a_{i+1}$ while keeping the inequalities $w_1(2i-1) < w_1 (2i)$ and $w_1 (2i+1) < w_1 (2i+2)$.
Thus, we can also assume that $w_1 (1) < w_1 (3) < \cdots < w_1 (2p-1)$.
Let $j \in \{1, \dots, q-1\}$ such that $w_1 (2p+j) = c_{j+1} > w_1(2p+j+1) = c_j$.
By applying the relations $y_{2p+j} X_{2i-1} = X_{2i-1} y_{2p+j}$ for $i \in \{1, \dots, p\}$, we can replace $w_1$ with $w_1 s_{2p+j}$, $h_1$ with $s_{2p+j}\,h_1\,s_{2p+j} \in \Z^n$, and $g_2$ with $g_2 s_{2p+j}$, and then $w_1 (2p+j) = c_j < w_1 (2p+j+1) = c_{j+1}$.
So, we can also assume that $w_1 (2p+1) < w_1 (2p+2) < \cdots < w_1(n)$.
We set $h_1 = e_1^{\nu_{1,1}} \cdots e_n^{\nu_{1,n}}$ and $h_2 = e_1^{\nu_{2,1}} \cdots e_n^{\nu_{2,n}}$.
Let $i \in \{1, \dots, p\}$.
By applying the relation $u_{2i-1}^{-1} X_{2i-1} =u_{2i} X_{2i-1}$, we can replace $\nu_{1,2i}$ with $0$ and $\nu_{1,2i-1}$ with $\nu_{1,2i-1}-\nu_{1,2i}$.
So, we can also assume that $\nu_{1,2i}=0$ for all $i \in \{1, \dots, p\}$.
Let $j \in \{1, \dots, q\}$.
By applying the relations $u_{2p+j} X_{2i-1} = X_{2i-1} u_{2p+j}$ for $i \in \{1, \dots, p\}$, we can replace $\nu_{1,2p+j}$ with $0$ and $\nu_{2,2p+j}$ with $\nu_{2,2p+j}-\nu_{1,2p+j}$.
Thus, we can also assume that $\nu_{1,2p+j}=0$ for all $j \in \{1, \dots, q\}$.
In conclusion, we can assume that $g_1 \in U_{1,p}$.

We can use the same argument to show that $g_2$ can be replaced with some $g_2' \in U_{2,p}$.
So, $\iota(g_1) \, B_p \, \iota(g_2^{-1}) \in \BB$.
This concludes the proof of Claim 4. 

{\it Claim 5.}
{\it Let $p \in \{0, 1, \dots, n_0\}$, $a, b \in \{1, \dots, n-1\}$ and $m \in \Z$, such that $a \le b$.
Then $X_a u_a^m y_{a+1} \cdots y_b B_p \in \Span_{R_0^a}(\BB)$.}

{\it Proof of Claim 5.}
Suppose $a \ge 2p+1$.
Then 
\begin{gather*}
X_a u_a^m y_{a+1} \cdots y_b B_p =
X_a B_p u_a^m y_{a+1} \cdots y_b =\\
X_a (y_{a-1} y_a) (y_{a-2} y_{a-1}) \cdots (y_{2p+1} y_{2p+2})(y_{2p+2} y_{2p+1}) \cdots (y_{a-1} y_{a-2}) (y_a y_{a-1}) B_p u_a^m y_{a+1} \cdots y_b =\\
(y_{a-1} y_a) (y_{a-2} y_{a-1}) \cdots (y_{2p+1} y_{2p+2})X_{2p+1} B_p (y_{2p+2} y_{2p+1}) \cdots (y_{a-1} y_{a-2}) (y_a y_{a-1}) u_a^m y_{a+1} \cdots y_b =\\
(y_{a-1} y_a) (y_{a-2} y_{a-1}) \cdots (y_{2p+1} y_{2p+2}) B_{p+1} (y_{2p+2} y_{2p+1}) \cdots (y_{a-1} y_{a-2}) (y_a y_{a-1}) u_a^m y_{a+1} \cdots y_b =\\
\iota (g_1) \, B_{p+1} \iota(g_2^{-1}) \in \BB\,,
\end{gather*}
where 
\begin{gather*}
g_1 = (s_{a-1} s_a) (s_{a-2} s_{a-1}) \cdots (s_{2p+1} s_{2p+2})\,,\\
g_2 = s_b \cdots s_{a+1} e_a^{-m} (s_{a-1} s_a) (s_{a-2} s_{a-1}) \cdots (s_{2p+1} s_{2p+2})\,.
\end{gather*}
Suppose $a \le 2p$ and $a$ is even.
Let $c$ such that $a=2c$.
Then
\begin{gather*}
X_a u_a^m y_{a+1} \cdots y_b B_p =
X_{2c} u_{2c+1}^{-m} y_{2c+1} \cdots y_b X_{2c-1} X_1\cdots X_{2c-3} X_{2c+1} \cdots X_{2p-1} =\\
X_{2c} X_{2c-1} u_{2c+1}^{-m} y_{2c+1} \cdots y_b X_1\cdots X_{2c-3} X_{2c+1} \cdots X_{2p-1} =\\
y_{2c-1} y_{2c} y_{2c} y_{2c-1} X_{2c} X_{2c-1} u_{2c+1}^{-m} y_{2c+1} \cdots y_b X_1 \cdots X_{2c-3} X_{2c+1} \cdots X_{2p-1} =\\
y_{2c-1} y_{2c} u_{2c+1} X_{2c-1} u_{2c+1}^{-m} y_{2c+1} \cdots y_b X_1 \cdots X_{2c-3} X_{2c+1} \cdots X_{2p-1} =\\
\iota (s_{2c-1} s_{2c} e_{2c+1}^{1-m} s_{2c+1} \cdots s_b) B_p \in \BB\,.
\end{gather*}
Suppose $a \le 2p$, $a$ is odd, and $a=b$.
Let $c$ such that $a=2c-1$.
Then 
\begin{gather*}
X_a u_a^m y_{a+1} \cdots y_b B_p =
X_{2c-1} u_{2c-1}^m X_{2c-1} X_1 \cdots X_{2c-3} X_{2c+1} \cdots X_{2p-1} =\\
z_{|m|} X_{2c-1} X_1 \cdots X_{2c-3} X_{2c+1} \cdots X_{2p-1} =
z_{|m|} B_p \in \Span_{R_0^a} (\BB)\,.
\end{gather*}
Suppose $a\le 2p$, $a$ is odd, and $a<b$.
Let $c$ such that $a=2c-1$.
Then 
\begin{gather*}
X_a u_a^m y_{a+1} \cdots y_b B_p =
X_{2c-1} u_{2c-1}^m y_{2c} y_{2c+1} \cdots y_b X_{2c-1} X_1 \cdots X_{2c-3} X_{2c+1} \cdots X_{2p-1} =\\
X_{2c-1} u_{2c}^{-m} y_{2c} X_{2c-1} y_{2c+1} \cdots y_b X_1 \cdots X_{2c-3} X_{2c+1} \cdots X_{2p-1} =\\
X_{2c-1} y_{2c} u_{2c+1}^{-m} X_{2c-1} y_{2c+1} \cdots y_b X_1 \cdots X_{2c-3} X_{2c+1} \cdots X_{2p-1} =\\
X_{2c-1} y_{2c} X_{2c-1} u_{2c+1}^{-m} y_{2c+1} \cdots y_b X_1 \cdots X_{2c-3} X_{2c+1} \cdots X_{2p-1} =\\
X_{2c-1} u_{2c+1}^{-m} y_{2c+1} \cdots y_b X_1 \cdots X_{2c-3} X_{2c+1} \cdots X_{2p-1} =
\iota( e_{2c+1}^{-m} s_{2c+1} \cdots s_b) B_p \in \BB\,.
\end{gather*}

{\it Claim 6.}
{\it Let $g \in G$ and $p \in \{0, 1, \dots, n_0\}$.
Then $X_1 \iota(g)\,B_p \in \Span_{R_0^a} (\BB)$.}

{\it Proof of Claim 6.}
We write $g$ in the form $g = h w$ with $h = e_1^{\nu_1} \cdots e_n^{\nu_n} \in \Z^n$ and $w \in \SSS_n$.
We have
\[
X_1 \iota(g) B_p =
X_1 u_1^{\nu_1} u_2^{\nu_2} \cdots u_n^{\nu_n} \iota(w) B_p =
\iota( e_3^{\nu_3} \cdots e_n^{\nu_n}) X_1 \iota(e_1^{\nu_1-\nu_2}w) B_p\,.
\]
Thus, by Claim 4, we can assume that $h=e_1^m$ with $m \in \Z$.
There exist $w_1 \in \langle s_3, \dots, s_{n-1} \rangle$, $a \in \{1, \dots, n-1\}$, and $b \in \{0, 1, \dots, n-1\}$ such that $w = w_1 s_2 s_3 \cdots s_a s_1 s_2 \cdots s_b$.
We have
\[
X_1 \iota(g) B_p =
X_1 u_1^m \iota(w_1) y_2 \cdots y_a y_1 \cdots y_b B_p =
\iota(w_1) X_1 u_1^m y_2 \cdots y_a y_1 \cdots y_b B_p\,.
\]
Thus, by Claim 4, we can assume that $w=s_2 \cdots s_a s_1 \cdots s_b$.
Suppose $a \le b$.
Then 
\begin{gather*}
X_1 \iota(g) B_p =
X_1 u_1^m y_2 \cdots y_a y_1 \cdots y_b B_p =\\
X_1 u_1^m (y_2 y_1) (y_3 y_2) \cdots (y_a y_{a-1}) y_a y_{a+1} \cdots y_b B_p =\\
(y_2 y_1) (y_3 y_2) \cdots (y_a y_{a-1}) X_a u_a^m y_a y_{a+1} \cdots y_b B_p =\\
(y_2 y_1) (y_3 y_2) \cdots (y_a y_{a-1}) X_a y_a u_{a+1}^m y_{a+1} \cdots y_b B_p =\\
(y_2 y_1) (y_3 y_2) \cdots (y_a y_{a-1}) X_a u_{a+1}^{m-1} y_{a+1} \cdots y_b B_p =\\
\iota ((s_2 s_1) (s_3 s_2) \cdots (s_a s_{a-1})) X_a u_a^{1-m} y_{a+1} \cdots y_b B_p \in
\Span_{R_0^a} (\BB)\,.
\end{gather*}
The last inclusion follows from Claim 4 and Claim 5.
Suppose $a > b$.
Then
\begin{gather*}
X_1 \iota(g) B_p =
X_1 u_1^m y_2 \cdots y_a y_1 \cdots y_b B_p =\\
X_1 u_1^m (y_2 y_1) (y_3 y_2) \cdots (y_{b+1} y_b) y_{b+2} \cdots y_a B_p =\\
\iota( (s_2 s_1) (s_3 s_2) \cdots (s_{b+1} s_b)) X_{b+1} u_{b+1}^m  y_{b+2} \cdots y_a B_p \in
\Span_{R_0^a} (\BB)\,.
\end{gather*}
Again, the last inclusion follows from Claim 4 and Claim 5.
This concludes the proof of Claim~6. 

As pointed out just before Claim 4, the following claim ends the proof of Proposition \ref{prop3_2}.

{\it Claim 7.}
{\it The set $\MM$ is contained in $\Span_{R_0^a} (\BB)$.}

{\it Proof of Claim 7.}
Let $Y \in \MM$.
By using the relations $y_i y_j X_i = X_j y_i y_j$ for $|i-j|=1$, we see that $Y$ can be written in the form $Y = \iota(g_0) X_1 \iota(g_1) X_1 \cdots X_1 \iota(g_k)$, where $k\ge 0$ and $g_0, g_1, \dots, g_k \in G$.
We argue by induction on $k$. 
The cases $k=0$ and $k=1$ follow directly from Claim 4.
So, we can assume that $k \ge 2$ and that the inductive hypothesis holds.
By the inductive hypothesis, $\iota(g_1) X_1 \cdots X_1 \iota(g_k) \in \Span_{R_0^a} (\BB)$.
Thus, we just have to show that $\iota (g_0) X_1 \iota (g_1') B_p \iota (g_2'^{-1}) \in \Span_{R_0^a} (\BB)$ for all $p \in \{0, 1, \dots, n_0\}$ and all $g_1', g_2' \in G$.
By Claim 6 we have $X_1 \iota (g_1') B_p \in \Span_{R_0^a} (\BB)$, hence, by Claim 4,  $\iota (g_0) X_1 \iota (g_1') B_p \iota (g_2'^{-1}) \in \Span_{R_0^a} (\BB)$.
This concludes the proof of Claim 7.
\end{proof}

Recall that the Temperley--Lieb algebra $\TL_n$ is the algebra over $R_0^f = \Z [z]$ defined by the presentation with generators $E_1, \dots, E_{n-1}$ and relations
\[
E_i^2 = z E_i \text{ for } 1 \le i \le n-1\,,\
E_i E_j = E_j E_i \text{ for } |i-j| \ge 2\,,\
E_i E_j E_i = E_i \text{ for } |i-j|=1\,.
\]
We see in the presentation given in Proposition \ref{prop3_2} that the relations $F_i^2 = z_0 F_i$, for $1 \le i \le n-1$, and $F_i F_j = F_j F_i$, for $|i-j| \ge 2$, hold in $\ATL_n$.
We also know that the relations $F_i F_j F_i = F_i$, for $|i-j|=1$, hold (see Claim 2 in the proof of Proposition \ref{prop3_2}).
So, we have a ring homomorphism $\iota: \TL_n \to \ATL_n$ which sends $z$ to $z_0$, and $E_i$ to $F_i$ for all $i \in \{1, \dots, n-1\}$.

\begin{prop}\label{prop3_3}
Let $n \ge 2$.
Then the above defined homomorphism $\iota: \TL_n \to \ATL_n$ is injective.
\end{prop}

\begin{proof}
We see from the presentations of $\VTL_n$ and $\ATL_n$ that there is a ring homomorphism $\varphi: \ATL_n \to \VTL_n$ which sends $z_m$ to $z$ for all $m \in \N$, $F_i$ to $E_i$ for all $i \in \{1, \dots, n-1\}$, $w_i$ to $v_i$ for all $i \in \{1, \dots, n-1\}$, and $t_j^{\pm 1}$ to $1$ for all $j \in \{1, \dots, n\}$.
By Proposition \ref{prop2_4} the composition $\varphi \circ \iota : \TL_n \to \VTL_n$ is injective, hence $\iota : \TL_n \to \ATL_n$ is also injective.
\end{proof}

Recall that $\ZZ^* = \{z_k\}_{k=1}^\infty$, $R^a = \Z [A^{\pm 1}, \ZZ^*]$, and that $R_0^a$ is embedded into $R^a$ via the identification $z_0 = -A^2 - A^{-2}$.
For each $n \ge 1$ we set $\ATL_n (R^a) =R^a \otimes \ATL_n$.
This is a $R^a$-algebra and a free $R^a$-module freely generated by $\FF_n$. 

\begin{thm}\label{thm3_4}
Let $n \ge 1$. 
There exists a homomorphism $\rho_n^a: R^a [\VB_n] \to \ATL_n (R^a)$ which sends $\sigma_i$ to $-A^{-2}\,1 - A^{-4} F_i$ and $\tau_i$ to $w_i$ for all $i \in \{1, \dots, n-1\}$.
\end{thm}

\begin{proof}
The proof is almost identical to that of Theorem \ref{thm2_5}.
We set $S_i = -A^{-2}\,1 - A^{-4} F_i$.
It is easily checked as in the proof of Theorem \ref{thm2_5} that $(-A^{-2}\,1-A^{-4}F_i)(-A^2\,1-A^4F_i)=1$, hence $S_i$ is invertible and $S_i^{-1} = -A^2\,1-A^4F_i$.
For $i\in\{1,\dots,n-1\}$ we have $w_i^2=1$, hence $w_i$ is also invertible.
It remains to see that the following relations hold. 
\begin{gather*}
w_i^2 = 1\,, \quad \text{for } 1 \le i \le n-1\,,\\
S_i S_j = S_j S_i\,,\ w_i w_j = w_j w_i\,,\ w_i S_j = S_j w_i\,,\quad \text{for } |i-j| \ge 2\,,\\
S_i S_j S_i = S_j S_i S_j\,,\ w_i w_j w_i = w_j w_i w_j\,,\ w_i w_j S_i = S_j w_i w_j\,,\quad\text{for } |i-j| = 1\,.
\end{gather*}
The only relation which does not follow directly from the presentation of $\ATL_n(R^a)$ is $S_i S_j S_i = S_j S_i S_j$, for $|i-j |= 1$.
But the latter can be proved in the same way as in the proof of Theorem~\ref{thm2_5}.
\end{proof}

\begin{rem}
\begin{itemize}
\item[(1)]
The sequence $\{ \rho_n^a: R^a [\VB_n] \to \ATL_n(R^a)\}_{n=1}^\infty$ is compatible with the tower of algebras $\{\ATL_n(R^a)\}_{n=1}^\infty$.
\item[(2)]
As in the case of virtual Temperley--Lieb algebras (see Section \ref{sec2}), setting $\rho_n^a (\sigma_i) = -A^{-2}\,1 - A^{-4} F_i$ instead of $\rho_n^a (\sigma_i) = A\,1+A^{-1} F_i$ allows to include in $\rho_n^a$ the corrective with the writhe and to define directly the arrow polynomial without passing through the arrow Kauffman bracket. 
\item[(3)]
For each $n \ge 1$ and $\beta \in B_n$ we have $\rho_n^a (\beta) = \rho_n^f (\beta) \in \TL_n(R^f)$.
\end{itemize}
\end{rem}

Recall that $V_n = \{0, 1\} \times \{1, \dots, n\}$ is ordered by $(0,1) < (0,2) < \cdots < (0,n) < (1,n) < \cdots <(1,2) < (1,1)$. 
Let $E$ be a flat virtual tangle.
A \emph{cycle} of length $\ell$ in the \emph{closure} $\hat E$ of $E$ is a $\ell$-tuple $\hat \gamma = (\gamma_1, \dots, \gamma_\ell)$ in $E$, where $\gamma_i=\{ (a_i,b_{i-1}), (c_i,b_i) \}$, satisfying the following properties.
\begin{itemize}
\item[(a)]
$a_{i+1} = 1$ if $c_i = 0$, and $a_{i+1}=0$ if $c_i = 1$, for all $i \in \{1, \dots, \ell-1\}$.
\item[(b)]
$b_0=b_\ell < b_i$ for all $i \in \{1, \dots, \ell-1\}$, $a_1=0$, and $c_\ell = 1$.
\end{itemize}
Let $F=(E,f)$ be an arrow flat $n$-tangle. 
We define the \emph{parity} of an element $\alpha = \{ (a,b), (c,d) \} \in E$ by $\varpi(\alpha)=-1$ if $a=c$, and $\varpi (\alpha) = 1$ if $a \neq c$.
Let $\hat\gamma=(\gamma_1, \dots, \gamma_\ell)$ be a cycle in $\hat E$.
We write $\gamma_i = \{(a_i, b_{i-1}), (c_i, b_i)\}$ for all $i$, and we define the \emph{cumulated parity} of $\gamma_i$ relative to $\hat \gamma$ by $\varpi^c (\gamma_i) = \prod_{j=1}^{i-1} \varpi(\gamma_j)$ if $(a_i,b_{i-1}) < (c_i,b_i)$, and $\varpi^c(\gamma_i) = \prod_{j=1}^{i} \varpi(\gamma_j)$ if $(c_i, b_i) < (a_i, b_{i-1})$. 
Then we set 
\[
h(\hat\gamma)=\sum_{i=1}^\ell\varpi^c (\gamma_i)\, f (\gamma_i)\,.
\]
It is easily seen that $h(\hat\gamma)$ is an even number.
The \emph{number of zigzags} of $\hat \gamma$ is defined by $\zeta (\hat \gamma)=\frac{|h(\hat\gamma)|}{2}$.
Let $\hat \gamma_1, \dots, \hat \gamma_m$ be the cycles of $\hat E$.
Then we set 
\[
T_n'^a (F)=z_{\zeta(\hat \gamma_1)}\,z_{\zeta(\hat \gamma_2)} \cdots z_{\zeta(\hat \gamma_m)}\,.
\]
We define $T_n'^a: \ATL_n(R^a) \to R^a$ by extending linearly the map $T_n'^a: \FF_n \to R^a$.

\begin{expl}
The closure of the arrow flat tangle of Figure \ref{fig3_5} is illustrated in Figure \ref{fig3_8}.
Here $\hat E$ has a unique cycle $\hat \gamma$ and $h(\hat \gamma)=1-(-6)-3 = 4$, hence $T_n'^a(F)=z_2$.
\end{expl}

\begin{figure}[ht!]
\begin{center}
\includegraphics[width=4cm]{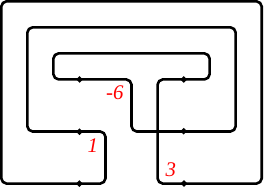}
\caption{Closure of an arrow flat tangle}\label{fig3_8}
\end{center}
\end{figure}

\begin{thm}\label{thm3_5}
The sequence $\{T_n'^a : \ATL_n(R^a) \to R^a\}_{n=1}^\infty$ is a Markov trace.
\end{thm}

\begin{proof}
For each $n \ge 2$ and each $i \in \{1, \dots, n-1\}$ we set $S_i=-A^{-2} \,1 - A^{-4} F_i$.
We need to prove that the following equalities hold.
\begin{itemize}
\item[(1)]
$T_n'^{a} (xy) = T_n'^{a} (yx)$ for all $n \ge 1$ and all $x,y\in \ATL_n (R^a)$;
\item[(2)]
$T_n'^{a} (x) = T_{n+1}'^{a} (x S_n)= T_{n+1}'^{a} (xS_n^{-1}) = T_{n+1}'^{a} (xw_n)$ for all $n \ge 1$ and all $x \in \ATL_n (R^a)$;
\item[(3)]
$T_n'^{a} (x) = T_{n+1}'^{a} (x S_n^{-1} w_{n-1} S_n)$ for all $n \ge 2$ and all $x \in \ATL_n (R^a)$;
\item[(4)]
$T_n'^{a} (x) = T_{n+1}'^a (x w_n w_{n-1} S_{n-1} w_n S_{n-1}^{-1} w_{n-1} w_n)$ for all $n \ge 2$ and all $x \in \ATL_n (R^a)$.
\end{itemize}

{\it Proof of (1).}
We can assume that $x = F = (E,f)$ and $y = F' = (E',f')$ are arrow flat $n$-tangles.  
As in the definition of the multiplication in $\ATL_n$, for $\alpha \in E \sqcup E'$, we set $f^* (\alpha) = f (\alpha)$ if $\alpha \in E$ and $f^* (\alpha) = f' (\alpha)$ if $\alpha \in E'$.
A \emph{long cycle} of length $2p$ in $E \sqcup E'$ is a $2p$-tuple $\hat \gamma = (\gamma_1, \dots, \gamma_{2p})$ in $E \sqcup E'$, where $\gamma_i = \{ (a_i, b_{i-1}), (c_i, b_i) \}$, satisfying the following properties. 
\begin{itemize}
\item[(a)]
$\gamma_i \in E$ if $i$ is odd, $\gamma_i \in E'$ if $i$ is even, $a_{i+1}=1$ if $c_i=0$, $a_{i+1} = 0$ if $c_i = 1$ (The indices are taken in $\{1, \dots, 2p\}$ modulo $2p$. In particular, $b_{2p}=b_0$).
\item[(b)]
$(a_1,b_0)  < (a_i, b_{i-1})$ for all $i \in \{3, 5, \dots, 2p-1\}$ and $(a_1,b_0) < (c_i,b_i)$ for all $i \in \{1, 3, \dots, 2p-1\}$.
\end{itemize}
The \emph{cumulated parity} of $\gamma_i$ relative to $\hat \gamma$ is $\varpi^c (\gamma_i) = \prod_{j=1}^{i-1} \varpi (\gamma_j)$ if $(a_i,b_{i-1}) < (c_i,b_i)$ and $\varpi^c (\gamma_i) = \prod_{j=1}^i \varpi (\gamma_j)$ if $(a_i, b_{i-1}) > (c_i, b_i)$.
We set
\[
h (\hat \gamma) = \sum_{i=1}^{2p} \varpi^c (\gamma_i) \, f^* (\gamma_i) \,.
\]
We see that $h (\hat \gamma)$ is an even number.
Then we define the \emph{number of zigzags} of $\hat \gamma$ by $\zeta (\hat \gamma) = \frac{|h (\hat \gamma)|}{2}$.
Let $\hat \gamma_1, \dots, \hat \gamma_m$ be the long cycles of $E \sqcup E'$.
Then
\[
T_n'^a (FF') = z_{\zeta (\hat \gamma_1)} z_{\zeta (\hat \gamma_2)} \cdots z_{\zeta (\hat \gamma_m)} \,.
\]
Let $\hat \gamma = (\gamma_1, \dots, \gamma_{2p})$ be a long cycle of $E \sqcup E'$.
There exists a unique long cycle $\hat \gamma'$ of $E' \sqcup E$ of one the following forms
\[
(\gamma_i, \gamma_{i+1}, \dots, \gamma_{2p}, \gamma_1, \gamma_2, \dots, \gamma_{i-1}) \text{ or }
(\gamma_i, \gamma_{i-1}, \dots, \gamma_1, \gamma_{2p}, \gamma_{2p-1}, \dots, \gamma_{i+1})\,,
\]
with $i \in \{2, 4, \dots, 2p\}$. 
In addition, each long cycle of $E' \sqcup E$ is of this form, and $h (\hat \gamma') = \pm h (\hat \gamma)$, hence $\zeta (\hat \gamma') = \zeta (\hat \gamma)$.
We conclude that $T_n'^a (FF') = T_n'^a (F'F)$.

{\it Proof of (2).}
From now on the proof of Theorem \ref{thm3_5} is almost identical to that of Theorem \ref{thm2_6}. 
We can assume that $x=F=(E,f)$ is an arrow flat $n$-tangle.
We see in Figure \ref{fig3_9} that the following equalities hold.
\[
T_{n+1}'^a (F) = z_0 \, T_n'^a (F)\,,\ T_{n+1}'^a (FF_n) = T_n'^a (F) \,, \ T_{n+1}'^a (Fw_n) = T_n'^a (F)\,.
\]
Recall that $S_n^{-1} = -A^2 \, 1 - A^4 F_n$ (see the proof of Theorem \ref{thm3_4}).
It follows that
\begin{gather*}
T_{n+1}'^a (FS_n )=
-A^{-2} T_{n+1}'^a (F) - A^{-4} T_{n+1}'^a (FF_n)=
(-A^{-2} z_0 - A^{-4}) T_n'^a (F)=
T_n'^a(F)\,,\\
T_{n+1}'^a (FS_n^{-1}) =
-A^2 T_{n+1}'^a (F) -A^4 T_{n+1}'^a (FF_n) =
(-A^2 z_0 - A^4) T_n'^a(F)=
T_n'^a(F)\,.
\end{gather*}

\begin{figure}[ht!]
\begin{center}
\includegraphics[width=13.2cm]{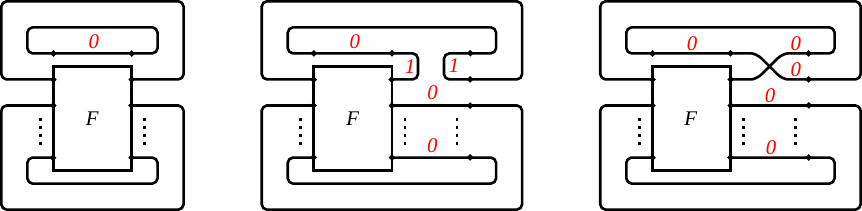}
\caption{$n+1$-closure of $F$, $FF_n$, and $Fw_n$}\label{fig3_9}
\end{center}
\end{figure}

{\it Proof of (3).}
We can again assume that $x = F = (E,f)$ is a flat $n$-tangle.
We have
\begin{gather*}
F S_n^{-1} w_{n-1} S_n =
F (-A^2\,1-A^4 F_n) w_{n-1} (-A^{-2}\,1-A^{-4} F_n)=\\
F w_{n-1} + A^{2} F F_n w_{n-1} + A^{-2} F w_{n-1} F_n + F F_n w_{n-1} F_n =\\
F w_{n-1} + A^{2} F F_n w_{n-1} + A^{-2} F w_{n-1} F_n + F F_n\,.
\end{gather*}
By the above, it follows that
\begin{gather*}
T_{n+1}'^a (F S_n^{-1} w_{n-1} S_n) =\\
T_{n+1}'^a (F w_{n-1}) + A^{-2} T_{n+1}'^a (F w_{n-1} F_n) + A^{2} T_{n+1}'^a (F F_n w_{n-1}) + T_{n+1}'^a (F F_n) =\\
z_0 T_n'^a (F w_{n-1}) + A^{-2} T_n'^a (F w_{n-1}) + A^{2} T_{n+1}'^a (w_{n-1} F F_n) + T_n'^a (F) =\\
(-A^{2} - A^{-2}) T_n'^a (F w_{n-1}) + A^{-2} T_n'^a (F w_{n-1}) + A^{2} T_n'^a (w_{n-1} F) + T_n'^a (F) =\\
(-A^{2} - A^{-2}) T_n'^a (F w_{n-1}) + A^{-2} T_n'^a (F w_{n-1}) + A^{2} T_n'^a (F w_{n-1}) + T_n'^a (F) =
T_n'^a (F)\,.
\end{gather*}

{\it Proof of (4).}
We can again assume that $x = F = (E,f)$ is an arrow flat $n$-tangle.
Then 
\begin{gather*}
F w_n w_{n-1} S_{n-1} w_n S_{n-1}^{-1} w_{n-1} w_n =\\
F w_n w_{n-1} (-A^{-2}\,1 - A^{-4}F_{n-1}) w_n (-A^{2}\,1 - A^{4}F_{n-1}) w_{n-1} w_n =\\
(F w_n w_{n-1} w_n w_{n-1} w_n) + 
A^{2} (F w_n w_{n-1} w_n F_{n-1} w_{n-1} w_n) +\\
A^{-2} (F w_n w_{n-1} F_{n-1} w_n w_{n-1} w_n) +
(F w_n w_{n-1} F_{n-1} w_n F_{n-1} w_{n-1} w_n) =\\
(F w_{n-1} w_n w_{n-1} w_{n-1} w_n) + 
A^{2} (F w_n F_n w_{n-1} w_n w_{n-1} w_n) +\\
A^{-2} (F w_n w_{n-1} w_n w_{n-1} F_n w_n) +
(F w_n w_{n-1} F_{n-1} w_{n-1} w_n) =\\
(F w_{n-1}) + 
A^{2} (F w_n F_n w_{n-1} w_{n-1} w_n w_{n-1}) +\\
A^{-2} (F w_{n-1} w_n w_{n-1} w_{n-1} F_n w_n) +
(F w_n t_{n-1}^{-1} F_{n-1} t_{n-1} w_n) =\\
(F w_{n-1}) + 
A^{2} (F w_n F_n w_n w_{n-1}) +
A^{-2} (F w_{n-1} w_n F_n w_n) +
(F t_{n-1}^{-1} w_n F_{n-1} w_n t_{n-1}) =\\
(F w_{n-1}) + 
A^{2} (F t_{n}^{-1} F_n t_{n} w_{n-1}) +
A^{-2} (F w_{n-1} t_{n}^{-1} F_n t_{n}) +
(F t_{n-1}^{-1} w_{n-1} F_n w_{n-1} t_{n-1})\,.
\end{gather*}
By the above, it follows that
\begin{gather*}
T_{n+1}'^{a} (F w_n w_{n-1} S_{n-1} w_n S_{n-1}^{-1} w_{n-1} w_n) =\\
T_{n+1}'^{a} (F w_{n-1}) + 
A^{2} T_{n+1}'^{a} (F t_{n}^{-1} F_n t_{n} w_{n-1}) +\\
A^{-2} T_{n+1}'^{a} (F w_{n-1} t_{n}^{-1} F_n t_{n}) +
T_{n+1}'^{a} (F t_{n-1}^{-1} w_{n-1} F_n w_{n-1} t_{n-1} ) =\\
z_0 T_n'^{a} (F w_{n-1}) + 
A^{2} T_{n+1}'^{a} (t_{n} w_{n-1} F t_{n}^{-1} F_n) +\\
A^{-2} T_{n+1}'^{a} (t_{n} F w_{n-1} t_{n}^{-1} F_n) +
T_{n+1}'^{a} (w_{n-1} t_{n-1} F t_{n-1}^{-1} w_{n-1} F_n) =\\
(-A^{2}-A^{-2}) T_n'^{a} (F w_{n-1}) + 
A^{2} T_n'^{a} (t_{n} w_{n-1} F t_{n}^{-1}) +\\
A^{-2} T_n'^{a} (t_{n} F w_{n-1} t_{n}^{-1}) +
T_n'^{a} (w_{n-1} t_{n-1} F t_{n-1}^{-1} w_{n-1}) =\\
(-A^{2}-A^{-2}) T_n'^{a} (F w_{n-1}) + 
A^{2} T_n'^{a} (F t_{n}^{-1} t_{n} w_{n-1}) +\\
A^{-2} T_n'^{a} (F w_{n-1} t_{n}^{-1} t_{n}) +
T_n'^{a} (F t_{n-1}^{-1} w_{n-1} w_{n-1} t_{n-1}) =\\
(-A^{2}-A^{-2}) T_n'^{a} (F w_{n-1}) + 
A^{2} T_n'^{a} (F w_{n-1}) +
A^{-2} T_n'^{a} (F w_{n-1}) +
T_n'^{a} (F) =
T_n'^{a} (F)\,.\quad\proved
\end{gather*}
\end{proof}

\begin{corl}\label{corl3_6}
For each $n \ge 1$ we set $T_n^a = T_n'^a \circ \rho_n^a :R^a [\VB_n] \to R^a$.
Then the sequence $\{T_n^a: R^a [\VB_n] \to R^a\}_{n=1}^\infty$ is a Markov trace.
\end{corl}

Recall that $R^f = \Z [A^{\pm 1}]$ is a subring of $R^a = \Z [A^{\pm 1}, \ZZ^*]$ and that, for $n \ge 1$, we have an embedding $\iota_n : \TL_n (R^f) \to \ATL_n (R^a)$ which sends $E_i$ to $F_i$ for all $i \in \{1, \dots, n-1\}$ (see Proposition \ref{prop3_3}).
The next proposition is a version of Theorem \ref{thm3_1}\,(2) in terms of Temperley--Lieb algebras. 

\begin{prop}\label{prop3_7}
Let $n \ge 1$.
Then $T_n'^a (\iota_n(x)) = T_n'^f (x) \in R^f$ for all $x \in \TL_n (R^f)$.
\end{prop}

We denote by $\EE_n^\nu$ the set of flat virtual $n$-tangles $E \in \EE_n$ such that, for each $\alpha = \{ (a,b), (c,d) \} \in E$, $d-b$ is odd if $a=c$, and $d-b$ is even if $a \neq c$. 
We denote by $\VTL_n^\nu$ the $R_0^f$-submodule of $\VTL_n$ spanned by $\EE_n^\nu$.
For each $E \in \EE_n^\nu$ we define $\iota_n^\nu (E) =(E,f) \in \FF_n$ as follows.
Let $\alpha=\{ (a,b), (c,d) \} \in E$.
If $a=c$, we set $f(\alpha)=|d-b|$, and if $a=0$ and $c=1$, we set $f(\alpha)=d-b$.
We define $\iota_n^\nu : \VTL_n^\nu \to \ATL_n$ by identifying $z$ with $z_0$ and extending linearly the map $\iota_n^\nu : \EE_n^\nu \to \FF_n$.
The key point in the proof of Proposition \ref{prop3_7} is the following. 

\begin{lem}\label{lem3_8}
Let $n \ge 1$.
\begin{itemize}
\item[(1)]
$\VTL_n^\nu$ is 	a subalgebra of $\VTL_n$ and $\iota_n^\nu : \VTL_n^\nu \to \ATL_n$ is a ring homomorphism.
\item[(2)]
Let $\VTL_n^\nu (R^f) = R^f \otimes \VTL_n^\nu$, and let $\iota_n^\nu : \VTL_n^\nu (R^f) \to \ATL_n (R^a)$ be the homomorphism induced by $\iota_n^\nu : \VTL_n^\nu \to \ATL_n$.
Then $T_n'^a (\iota_n^\nu (x)) = T_n'^f (x) \in R^f$ for all $x\in \VTL_n^\nu (R^f)$.
\end{itemize}
\end{lem}

\begin{proof}
Let $E, E' \in \EE_n^\nu$.
Set $\iota_n^\nu (E) = (E,f)$, $\iota_n^\nu (E') = (E', f')$, and $\iota_n^\nu (E) * \iota_n^\nu (E') = (E * E', g)$.
We first show that $E * E' \in \EE_n^\nu$ and that $(E * E', g)=\iota_n^\nu (E * E')$.
Let $\hat \alpha$ be an arc in $E \sqcup E'$ of length $\ell$, and let $\partial \hat \alpha = \{ x', y' \}$, with $x'=(a,b) < y'=(c,d)$.

Suppose $a = c = 0$ and $\ell = 1$.
Then $\partial \hat \alpha = \alpha_1 = \{ (0,b), (0,d) \}$, $d-b$ is odd, and $g (\partial \hat \alpha) = d-b = |d-b|$.

Suppose $a = c = 0$ and $\ell >1$.
Then $\ell$ is odd, say $\ell = 2 p + 1$.
There exists a sequence $a_0, a_1, \dots, a_{2p+1}$ in $\{ 1, \dots, n \}$ such that $\hat \alpha$ is of the form $\hat \alpha = (\alpha_0, \beta_1, \gamma_2, \dots, \gamma_{2p-2}, \beta_{2p-1}, \alpha_{2p} )$, where $\alpha_0 = \{ (0,a_0), (1, a_1) \} \in E$, $\beta_i = \{(0, a_i), (0, a_{i+1}) \} \in E'$ for all $i\in\{1,3, \dots, 2p-1\}$, $\gamma_i = \{ (1,a_i), (1,a_{i+1}) \} \in E$ for all $i \in \{2,4, \dots, 2p-2 \}$, and $\alpha_{2p} = \{ (1, a_{2p}), (0,a_{2p+1}) \} \in E$.
We also have $x'=(0,a_0)$, $y'=(0,a_{2p+1})$, and $a_0=b < a_{2p+1}=d$.
The numbers $a_1 - a_0$ and $a_{2p+1} - a_{2p}$ are even, and the number $a_{i+1} - a_i$ is odd for every $i \in \{1, \dots, 2p-1\}$, hence $a_{2p+1} - a_0=d-b$ is odd. 
We have $\varpi^c (\alpha_0)=1$, hence $\varpi^c (\alpha_0)\, f^*(\alpha_0) = a_1 - a_0$.
Let $i \in \{1,3, \dots, 2p-1\}$.
We have $\varpi^c (\beta_i) = 1$ if $a_i < a_{i+1}$ and $\varpi^c (\beta_i) = -1$ if $a_i > a_{i+1}$.
In both cases we have $\varpi^c (\beta_i)\, f^* (\beta_i) = a_{i+1} - a_i$.
Let $i \in \{2, 4, \dots, 2p-2 \}$.
We have $\varpi^c(\gamma_i) = 1$ if $a_i < a_{i+1}$ and $\varpi^c(\gamma_i) = -1$ if $a_i > a_{i+1}$.
In both cases we have $\varpi^c (\gamma_i)\, f^* (\gamma_i) = a_{i+1} - a_i$.
We have $\varpi^c (\alpha_{2p})=-1$, hence $\varpi^c (\alpha_{2p})\, f^*(\alpha_{2p}) = a_{2p+1} - a_{2p}$.
So, $g (\hat \alpha) = g(\partial \hat \alpha) = a_{2p+1}-a_0=d-b=|d-b|$.

Suppose $a=c=1$ and $\ell = 1$.
Then $\partial \hat \alpha = \alpha_1 = \{ (1,b), (1,d) \}$, $b-d$ is odd, and $g (\partial \hat \alpha) = b-d = |b-d|$.
Suppose $a=c=1$ and $\ell > 1$.
Then we show in the same way as for the case $a=c=0$ and $\ell > 1$ that $b - d$ is odd and $g (\hat \alpha) = g (\partial \hat \alpha) = b-d = |b-d|$.

Suppose $a=0$ and $c=1$.
Then $\ell \ge 2$ and $\ell$ is even, say $\ell = 2p+2$.
There exists a sequence $a_0, a_1, \dots, a_{2p+2}$ in $\{ 1, \dots, n \}$ such that $\hat \alpha$ is of the form $\hat \alpha = (\alpha_0, \beta_1, \gamma_2, \dots, \beta_{2p-1}, \gamma_{2p}, \delta_{2p+1})$, where $\alpha_0=\{ (0,a_0), (1, a_1) \} \in E$, $\beta_i = \{ (0,a_i), (0, a_{i+1}) \} \in E'$ for all $i \in \{1, 3, \dots, 2p-1 \}$, $\gamma_i = \{ (1, a_i), (1, a_{i+1}) \} \in E$ for all $i \in \{2, 4, \dots, 2p \}$, and $\delta_{2p+1} = \{ (0, a_{2p+1}), (1, a_{2p+2}) \} \in E'$.
We also have $x'=(0,b)=(0,a_0)$ and $y'=(1,d)=(1,a_{2p+2})$.
The numbers $a_1 - a_0$ and $a_{2p+2} - a_{2p+1}$ are even and $a_{i+1} - a_i$ is odd for each $i \in \{1, \dots, 2p\}$, hence $d-b = a_{2p+2} - a_0$ is even.
We have $\varpi^c (\alpha_0)=1$, hence $\varpi^c (\alpha_0)\, f^*(\alpha_0) = a_1 - a_0$.
Let $i \in \{1,3, \dots, 2p-1\}$.
We have $\varpi^c (\beta_i) = 1$ if $a_i < a_{i+1}$ and $\varpi^c (\beta_i) = -1$ if $a_i > a_{i+1}$.
In both cases we have $\varpi^c (\beta_i)\, f^* (\beta_i) = a_{i+1} - a_i$.
Let $i \in \{2, 4, \dots, 2p \}$.
We have $\varpi^c(\gamma_i) = 1$ if $a_i < a_{i+1}$ and $\varpi^c(\gamma_i) = -1$ if $a_i > a_{i+1}$.
In both cases we have $\varpi^c (\gamma_i)\, f^* (\gamma_i) = a_{i+1} - a_i$.
We have $\varpi^c (\delta_{2p+1})=1$, hence $\varpi^c (\delta_{2p+1})\, f^*(\delta_{2p+1}) = a_{2p+2} - a_{2p+1}$.
So, $g (\hat \alpha) = g(\partial \hat \alpha) = a_{2p+2}-a_0= d-b$.

The above shows that $E * E' \in \EE_n^\nu$ and $\iota_n^\nu (E * E') = \iota_n^\nu (E) * \iota_n^\nu (E')$.

Let $\hat \gamma$ be a cycle in $E \sqcup E'$ of length $\ell$.
Then $\ell \ge 2$ and $\ell$ is even, say $\ell = 2p$.
There exists a sequence $a_0, a_1, \dots, a_{2p}$ in $\{1, \dots, n\}$ such that $\hat \gamma$ is of the form $\hat \gamma = (\alpha_1, \beta_2, \dots, \alpha_{2p-1}, \beta_{2p})$, where $\alpha_i = \{ (1, a_{i-1}), (1,a_i) \} \in E$ for all $i \in \{1, 3, \dots, 2p-1 \}$ and $\beta_i = \{ (0, a_{i-1}), (0, a_i) \} \in E'$ for all $i \in \{2, 4, \dots, 2p \}$.
We also have $a_{2p}=a_0$.
Let $i \in \{1, 3, \dots, 2p-1\}$.
We have $\varpi^c (\alpha_i) = 1$ if $a_{i-1} > a_i$ and $\varpi^c (\alpha_i) = -1$ if $a_{i-1} < a_i$.
In both cases we have $\varpi^c(\alpha_i) \, f^* (\alpha_i) = a_{i-1} - a_i$.
Let $i \in \{ 2, 4, \dots, 2p \}$.
We have $\varpi^c (\beta_i) = 1$ if $a_{i-1} > a_i$ and $\varpi^c (\beta_i) = -1$ if $a_{i-1} < a_i$.
In both cases we have $\varpi^c(\beta_i) \, f^* (\beta_i) = a_{i-1} - a_i$.
Thus, $h (\hat \gamma) = a_0 - a_{2p} = 0$, hence $\zeta (\hat \gamma) = 0$.
So, if $m$ is the number of cycles in $E \sqcup E'$, then $\iota_n^\nu (E)\, \iota_n^\nu (E') = z_0^m\, (\iota_n^\nu (E) * \iota_n^\nu (E'))$, hence $\iota_n^\nu (E)\, \iota_n^\nu (E') = \iota_n^\nu (E E')$.
This concludes the proof of the first part of the lemma.

Let $E \in \EE_n^\nu$ and let $(E,f) = \iota_n^\nu (E)$.
In order to prove the second part of the lemma, it suffices to show that, if $\hat \gamma$ is a cycle of $\hat E$, then $\zeta (\hat \gamma) = 0$, that is, $h(\hat \gamma) = 0$.
Let $\hat \gamma$ be a cycle of $\hat E$.
Then $\hat \gamma$ is of the form 
\begin{gather*}
\hat \gamma = (\alpha_{1,1}, \dots, \alpha_{1,p_1}, \beta_1, \gamma_{1,1}, \dots, \gamma_{1,q_1}, \delta_1, \dots, \alpha_{r,1}, \dots, \alpha_{r,p_r}, \beta_r, \gamma_{r,1}, \dots, \gamma_{r,q_r},\\ \delta_r, \alpha_{r+1,1}, \dots, \alpha_{r+1,p_{r+1}})\,,
\end{gather*}
where
\begin{gather*}
\alpha_{i,j} = \{ (0, a_{i,j-1}), (1, a_{i,j} ) \}\,,\
\beta_i = \{ (0,b_{i,0}), (0, b_{i,1}) \}\,,\
\gamma_{i,j} = \{ (1, c_{i,j-1}), (0, c_{i,j}) \}\,,\\ 
\delta_i = \{ (1, d_{i,0}), (1, d_{i,1}) \}\,,\ 
a_{i,p_i} = b_{i,0}\,,\
b_{i,1} = c_{i,0}\,,\
c_{i,q_i} = d_{i,0}\,,\
d_{i,1} = a_{i+1,0}\,,\
a_{r+1,p_{r+1}} = a_{1,0}\,.
\end{gather*}
We have $\varpi^c (\alpha_{i,j}) = 1$, hence $\varpi^c (\alpha_{i,j}) \, f (\alpha_{i,j}) = a_{i,j} - a_{i,j-1}$.
We have $\varpi^c (\beta_i) = -1$ if $b_{i,0} > b_{i,1}$ and $\varpi^c (\beta_i) = 1$ if $b_{i,0} < b_{i,1}$.
In both cases we have $\varpi^c (\beta_i) \, f (\beta_i) = b_{i,1} - b_{i,0}$.
We have $\varpi^c (\gamma_{i,j}) = -1$, hence $\varpi^c (\gamma_{i,j}) \, f (\gamma_{i,j}) = c_{i,j} - c_{i,j-1}$.
We have $\varpi^c (\delta_i) = -1$ if $d_{i,0} > d_{i,1}$ and $\varpi^c (\delta_i) = 1$ if $d_{i,0} < d_{i,1}$.
In both cases we have $\varpi^c (\delta_i) \, f (\delta_i) = d_{i,1} - d_{i,0}$.
Combining these equalities we get $h (\hat \gamma) = a_{r+1,p_{r+1}} - a_{1,0} = 0$, hence $\zeta (\hat \gamma) = 0$.
\end{proof}

\begin{proof}[Proof of Proposition \ref{prop3_7}]
We have $E_i \in \VTL_n^\nu (R^f)$ for all $i \in \{ 1, \dots, n-1 \}$, $\TL_n (R^f)$ is generated by $E_1, \dots, E_{n-1}$, and $\VTL_n^\nu (R^f)$ is a subalgebra of $\VTL_n(R^f)$, hence $\TL_n (R^f) \subset \VTL_n^\nu (R^f)$.
Moreover, since $\iota_n (E_i) = \iota_n^\nu (E_i)$ for all $i \in \{1, \dots, n-1\}$, we have $\iota_n (x) = \iota_n^\nu (x)$ for all $x \in \TL_n (R^f)$.
We conclude from Lemma \ref{lem3_8} that $T_n'^a (\iota_n (x)) = T_n'^f (x) \in R^f$ for all $x \in \TL_n (R^f)$.
\end{proof}

Now, it remains to prove the following. 

\begin{thm}\label{thm3_9}
Let $I^a: \VV\LL\to R^a$ be the invariant defined from the Markov trace of Corollary \ref{corl3_6}.
Then $I^a$ coincides with the arrow polynomial.
\end{thm}

\begin{proof}
The proof is similar to that of Theorem \ref{thm2_8}.
Let $\beta$ be a virtual braid on $n$ strands and let $\hat\beta$ be its closure.
It is easily seen that the relations $\langle\!\langle L \rangle\!\rangle = A \langle\!\langle L_1 \rangle\!\rangle + A^{-1} \langle\!\langle L_2 \rangle\!\rangle$ and $\langle\!\langle L \rangle\!\rangle = A^{-1} \langle\!\langle L_1 \rangle\!\rangle + A \langle\!\langle L_2\rangle\!\rangle$ in the definition of the arrow Kauffman bracket corresponds in terms of closed virtual braids to replacing each $\sigma_i$ with $A\,1+A^{-1}F_i$ and each $\sigma_i^{-1}$ with $A^{-1}\,1+A\,F_i$.
Once we have replaced each $\sigma_i$ with $A\,1+A^{-1}F_i$ and each $\sigma_i^{-1}$ with $A^{-1}\,1+A\,F_i$, we get a linear combination $\sum_{i=1}^\ell a_iF^{(i)}$, where $F^{(i)} = (E^{(i)},f^{(i)}) \in \FF_n$ and $a_i\in R^a$ for all $i$.
For $i\in \{1, \dots, \ell\}$ let $\hat \gamma_{i,1}, \dots, \hat \gamma_{i,m_i}$ be the cycles of $\hat E^{(i)}$.
Then 
\[
\langle\!\langle \hat \beta \rangle\!\rangle =\sum_{i=1}^\ell a_i z_{\zeta(\hat \gamma_{i,1})} z_{\zeta(\hat \gamma_{i,2})} \cdots z_{\zeta(\hat \gamma_{i,m_i})}\,. 
\]
Recall that $w : \VV \LL \to \Z$ denotes the writhe.
Let $\omega: \VB_n \to \Z$ be the homomorphism which sends $\sigma_i$ to $1$ and $\tau_i$ to $0$ for all $i \in \{1, \dots, n-1\}$.
Then $w (\hat \beta) = \omega (\beta)$, and therefore $\overrightarrow{f} (\hat \beta) = (-A^3)^{-\omega (\beta)} \langle\!\langle \hat \beta \rangle\!\rangle$.
So, in the above procedure, if we replace each $\sigma_i$ with $(-A^3)^{-1}(A\,1+A^{-1}F_i)=-A^{-2}\,1-A^{-4}F_i$ and each $\sigma_i^{-1}$ with $(-A^3)(A^{-1}\,1+A\,F_i)=-A^2\,1-A^4F_i$, then we get directly $\overrightarrow{f} (\hat \beta)$.
It is clear that this procedure also leads to $T_n^a (\beta)$.
\end{proof}



\end{document}